\documentclass[journal]{IEEEtran}   
\usepackage{multirow}
\usepackage{diagbox}
\usepackage{amssymb}
\usepackage{amsmath}
\usepackage{graphicx}
\usepackage{cite}
\usepackage{citesort}
\usepackage{subfigure}
\usepackage{graphicx,epstopdf}
\usepackage{epsfig}	
\usepackage{cite,graphicx,amsmath,amssymb}
\usepackage{comment}

\usepackage{amssymb}
\usepackage{amsmath}
\usepackage{cite}
\usepackage{url}
\usepackage{xcolor}
\usepackage{cite,graphicx,amsmath,amssymb}
\usepackage{subfigure}
\usepackage{citesort}
\usepackage{fancyhdr}
\usepackage{mdwmath}
\usepackage{mdwtab}
\usepackage{caption}
\usepackage{amsthm}
\usepackage{lipsum}
\UseRawInputEncoding 

\newtheorem{theorem}{Theorem}

\newtheorem{lemma}{Lemma}

\newtheorem{corollary}{Corollary}

\newtheorem{remark}{Remark}  

\makeatletter
\def\ScaleIfNeeded{%
\ifdim\Gin@nat@width>\linewidth \linewidth \else \Gin@nat@width
\fi } \makeatother

\begin{document}

\title{\Huge{Performance Analysis of Double Reconfigurable Intelligent Surfaces Assisted NOMA Networks  }}

\author{Xuehua\ Li,~\IEEEmembership{Member,~IEEE}, Xuanhao Lian, Xinwei~Yue,~\IEEEmembership{Senior Member,~IEEE}, Zhiping Lu, Chongwen Huang,~\IEEEmembership{Member,~IEEE}, Tianwei Hou, ~\IEEEmembership{Member,~IEEE}

\thanks{This work was supported in part by National Natural Science Foundation of China (Grant 62071052) and in part by Beijing Natural Science Foundation (Grant L222004). \emph{(Corresponding author: Xinwei Yue, Tianwei Hou.)}}
\thanks{X. Li, X. Lian, and X. Yue are with the Key Laboratory of Information and Communication Systems, Ministry of Information Industry and also with the Key Laboratory of Modern Measurement $\&$ Control Technology, Ministry of Education, Beijing Information Science and Technology University, Beijing 100101, China (email: \{lixuehua and xinwei.yue\}@bistu.edu.cn, lianxuanhao@gmail.com.)}
\thanks{Z. Lu is with the School of Information and Communication Engineering, Beijing University of Posts and Telecommunications, Beijing 100876, China and aslo with State Key Laboratory of Wireless Mobile Communications (CICT), Beijing 100191, China (luzp@bict.com).}
\thanks{C. Huang is with the College of Information Science and Electronic Engineering, Zhejiang University, Hangzhou, 310027, China (email: chongwenhuang@zju.edu.cn).}
\thanks{T. Hou is with the School of Electronic and Information Engineering, Beijing Jiaotong University, Beijing 100044, China, and also with the Institute for Digital Communications, Friedrich-Alexander Universit\"at Erlangen-N\"urnberg (FAU), 91054 Erlangen, Germany (email: twhou@bjtu.edu.cn).}

}


\maketitle

\begin{abstract}
This paper introduces double cascaded reconfigurable intelligent surfaces (RISs) to non-orthogonal multiple access (NOMA) networks over cascaded Rician fading and Nakagami-$m$ fading channels, where two kinds of passive RIS (PRIS) and active RIS (ARIS) are taken into consideration, called PRIS-ARIS-NOMA networks.
Additionally, new closed-form and asymptotic expressions for outage probability and ergodic data rate of two non-orthogonal users are derived with the imperfect/perfect successive interference cancellation schemes.
The scenario is modelled around two non-orthogonal users and focuses on analyzing their communication characteristics.
Based on the approximate results, the diversity orders and ergodic data rate slopes of two users are obtained in the high signal-to-noise ratios.
In addition, the system throughput of PRIS-ARIS-NOMA in delay-limited mode and delay-tolerant mode are discussed according to the outage probability and ergodic data rate.
The simulation results verify the correctness of the formulas and yields the following insights: 1) The outage behaviors of PRIS-ARIS-NOMA outperforms than that of PRIS-ARIS assisted orthogonal multiple access (OMA); 2)Use of PRIS-ARIS-NOMA is better than use of PRIS-ARIS-OMA in small transmit power threshold scenarios 3) By increasing the number of reflecting elements of RISs, the PRIS-ARIS-NOMA is able to achieve the enhanced outage performance; and 4) The PRIS-ARIS-NOMA has the higher ergodic data rate and system throughput than double PRISs-NOMA.
\end{abstract}
\section{Introduction}
With the recent advancements in science and technology, the focus of today's technological development has shifted towards achieving a higher quality of life.
The exponential growth in productivity has emerged as a guiding beacon, leading the way in technological progress.
Compared with the fifth-generation (5G) mobile communications, the sixth-generation (6G) is pinned with more hope for rapid development\cite{A_survey_of_5G}.
It aims to build super wireless broadband, super large-scale connectivity and extremely reliable communication capabilities \cite{Realizing_6G}.
And a prominent approach in current efforts to explore innovative, efficient and resource-friendly wireless network solutions for the future is reconfigurable intelligent surface (RIS) \cite{RISDAI,RISBasar,RISPan}.
The International Mobile Telecommunications 2030 has released the inaugural research report on RIS, signifying the formal inclusion of RIS within the realm of future research.

RIS is a relay-type planar array with programmable elements and serves as an artificial surface with the ability to adjust its electromagnetic properties \cite{Yuanwei_RIS2021}.
Each reflecting element is independent to adjust its signal reflecting amplitude and phase in response to the prevailing environment \cite{Beyond_Massive_MIMO}.
Currently, researches and discussions on RIS are in full swing, with both questions being raised and significant accomplishments being made \cite{Q_and_I}.
The authors of \cite{RIS_Channel_Estimation} studied the channel estimation of RIS and trained the signals behavior function by using various channel models.
And researchers in \cite{Communication_Models} discussed the study of RIS wireless communication performance optimization from a physical electromagnetic perspective.
On emphasis of the efficiency, the authors of \cite{How_many} explored the relationship between reflecting elements number and total rate in RIS-assisted multi-user networks for a balance solution.
Additionally, the impact of spatial channel correlation on the outage probability of RIS-assisted networks was investigated in \cite{OPRISSpatially}.
Moreover, a vehicle networking model was enhanced by integrating RIS in \cite{RISVehicular} and employing a selection concept scheme to determine the optimal RIS transmission.
For the line-of-sight transmission link studied by \cite{RISTAO}, the outage probability and ergodic capacity of RIS-aided networks were analyzed over Rician fading channels.

Compared to traditional RIS-assisted communication, various RIS methods have emerged including the active RIS (ARIS), which offers a distinct approach.
ARIS is an intelligent reflecting surface technology that offers a higher degree of flexibility and intelligence than traditional Passive RIS.
The main difference between ARIS and PRIS is the active adjustment of the reflecting elements.
These reflecting elements are usually composed of devices with variable impedance or adjustable devices, such as smart antenna elements, liquid crystal modulators, and so on.
By adjusting the impedance or other parameters of these reflective units, ARIS can adjust the characteristics of the reflecting surface in real time to adapt to different communication needs and environmental conditions.
In \cite{ARIS_design}, the design of ARIS was proposed to overcome the double-fading problem in passive RIS (PRIS) link and displays better performance than PRIS.
Simultaneously in \cite{ARIS_design2}, an ARIS-assisted networks were applied in a single-input multiple-output system and thermal noise optimization was analyzed with detailed formulations.
The asymptotic performance of ARIS with a max-rate problem were studied in \cite{ARISdesign_Dai} over the proposition of a joint transmit beamforming scheme.
In \cite{ARISdesign_Dai2}, a sub-connected ARIS scheme was proposed for energy consumption where several elements controlled signals amplification independently through the same power amplifier.
Furthermore the authors in \cite{ARIS_You} evaluated the superiority and inferiority of ARIS versus PRIS in terms of uplink and downlink communication scenarios.
And the authors of \cite{Yue_ARIS} studied the ARIS assisted NOMA networks and analyzed the performance of users with hardware impairments.
In addition to ARIS, another perspective involves the utilization of multiple RISs for assisting communication.
The authors of \cite{Novel_Multiple_RIS} suggested a novel multiple RIS-assisted networks over Nakagami-$m$ fading channels and analyzed the performance with the help of Gaussian distribution.
For using relay in multi-RISs system, the authors of \cite{Ergodic_RIS} characterized the achievable rate with the help of a decode-and-forward relay over Nakagami-$m$ fading channels.
Moreover in \cite{You_mul}, the best RIS was selected from multiple ARISs and PRISs to maximize the achievable data rate.
The authors of \cite{DoubleZheng} applied double RISs into MIMO system, involving the cooperative passive beamforming design from the inter-RIS channels.
In addition to the two RISs models mentioned above, a simultaneous transmitting and reflecting RIS (STARS) has been proposed \cite{addnew3,addnew4}, which has transmissive properties on top of reflective, and can consume a certain amount of resources to transmit signals through it, which improves the flexibility of RIS deployment.

As a pivotal technology for 6G, non-orthogonal multiple access (NOMA) is highly regarded by researchers due to its capability to enhance spectrum efficiency and throughput, especially in complex scenarios with numerous concurrent links \cite{YuanweiNOMA,NOMA_IoT,NOMA_energy}.
NOMA has ability to support multiple users' information which is linearly superposed at the same physical resource over different power levels \cite{Application_of_NOMA}.
And  a successive interference cancellation (SIC) will be carried out at the receiver to peel off the desired signals \cite{NOMAtech}.
The authors of \cite{Yingjie2023RISNOMAPLS} investigated the secrecy performance of NOMA networks, where the secrecy outage probability expressions of multiple users are derived.
And in \cite{Yue_full_half_NOMA}, cooperative NOMA networks were proposed where the nearby user acts as a relay to assist the distant user in transmitting signals.
Similarly in \cite{Liu_PLSNOMA}, the authors applied the NOMA framework in the direction of physical layer security and present the process of analyzing the secrecy outage probability.
For the direction of random access in \cite{SemiNOMA}, a framework of semi-grant-free NOMA networks was proposed for improving multi-user detection performance.
Current theoretical work that combines PRIS with NOMA has also been quite well referenced.
Presented in \cite{Ding_1bit}, PRIS-NOMA networks were suggested to serve additional edge users and designed with a phase shifting matrix calculation of "on-off" control.
And in \cite{yue2021RISNOMA}, the performance of PRIS-NOMA networks was analyzed with 1-bit coding where different relay baselines are set up to show superiority of PRIS-NOMA networks.
In \cite{Yuanwei_NOMA_RIS} for multiple PRISs aided multiple users NOMA networks, the authors analyzed users outage performance in scenarios with and without direct links between users and the source.
There are similarities between the research directions of \cite{You_mul} and \cite{Yuanwei_NOMA_RIS}, but in \cite{You_mul} the focus is on the selection of ARIS and PRIS as well as the optimal path assignment, while \cite{Yuanwei_NOMA_RIS} focuses on the system performance of RIS-assisted NOMA communication with multiple parallel transmissions.
Additionally in \cite{Yuanwei_RIS_and_multi-cell_NOMA}, the PRIS assisted multi-cell NOMA networks were analyzed by introducing a pass loss expression and found out that PRIS could adjust the SIC order in NOMA.
And the authors of \cite{STARS_Yue} analyzed the performance of NOMA networks aided by STARS to improve the channel variance between users in NOMA networks.
The authors of \cite{addnew1} proposed a cognitive unmanned aerial vehicle (UAV) network with NOMA aided by STARS, where STARS was utilized to assist the transmission and reflection of signals from a secondary UAV.
And for using deep reinforcement learning in UAV aided RIS-NOMA networks in \cite{addnew2}, the authors set RISs on the UAVs for the purpose of reshaping the wireless transmission path.
\subsection{Motivations and Contributions}
Building upon the previous sections, the literature on single RIS analysis has become abundant, and the fundamental research on RIS-assisted communication is fairly comprehensive.
Compared with a single RIS communication system, a communication system using two RISs has wider coverage, higher communication capacity, stronger anti-noise capability and more flexible beamforming capability, which can further improve the performance and efficiency of the communication system \cite{reviewer5r1,reviewer5r2}.
Furthermore, we took a keen interest in multi-RIS and would like to increase research on multi-RIS assisted NOMA communication networks at this stage.
In the recent period we note ARIS, which can better resist the impact of double-fading on communication compared to traditional RIS.
The authors in \cite{ARIS_design2} proposed a novel ARIS approach, where each elements of ARIS is deployed with an active load for signals reflection and amplification.
Consequently for a given power budget ARIS networks, enhanced performance can be achieved by increasing the number of reflecting elements and signal amplification gains \cite{Pancunhua}.
To address the path loss and channel control, the authors presented multi-RIS networks in \cite{You_mul} where ARISs and PRISs were deployed between the source and users for assistance.
In \cite{double_RIS_UAV}, one of the scenarios in which the sender and receiver communicate by means of two RISs as well as a relay is of great interest to us and we would like to explore the suitability of the modelling of this scenario for NOMA.
To the best of our knowledge, there are no previous studies to investigate PRIS-ARIS assisted NOMA networks over cascaded Rician fading and Nakagami-$m$ fading channels.
Thus we composed this paper to consider the feasibility of this scenario and to analyze it, which was our purpose and motivation for creating it.
More specifically, we investigate the performance of a pair of non-orthogonal users, a nearby user $D_n$ and a distant user $D_m$, for PRIS-ARIS-NOMA in terms of outage probability and ergodic data rate.
And PRIS-ARIS-NOMA is an abbreviated version of the PRIS and ARIS assisted NOMA communication networks
Noted that both ARIS and PRIS are part of the RIS technologies, so the system does not count as a Hybrid system.
Additionally, the outage probability and ergodic data rate of the two users are derived in detail.
According to the aforementioned explanations, the primary contributions of this manuscript are summarized as follows:
\begin{enumerate}
  \item We derive the closed-form expressions of outage probability for $D_n$ with imperfect SIC (ipSIC)/perfect SIC (pSIC) and $D_m$.
      We further deduce the asymptotic outage probabilities of $D_n$ with ipSIC/pSIC and $D_m$ to calculate the diversity orders of users.
      From the derivation results, it can be found that the users'diversity orders is related to the number of reflecting elements of ARIS.
  \item We set PRIS-ARIS-OMA and double PRISs-NOMA scenarios as the benchmark for comparison, then compare the outage performance among PRIS-ARIS-NOMA, double PRISs-NOMA and PRIS-ARIS-OMA scenarios.
      We set PRIS-ARIS-OMA and double PRISs-NOMA scenarios as the benchmark for comparison,
      We confirm that the outage performance with ARIS on the user side is better than that with PRIS, and the outage performance of PRIS-ARIS-NOMA  is better than that of PRIS-ARIS-OMA.
  \item We derive the the closed-form expressions of ergodic data rate for $D_n$ with ipSIC/pSIC and $D_m$, as well as ergodic data rate slopes of users in high signal-to-noise ratio (SNR) region. It can be confirmed that the ergodic data rate of PRIS-ARIS-NOMA is higher than that of double PRISs-NOMA and PRIS-ARIS-OMA.
  \item We evaluate the system throughputs of PRIS-ARIS-NOMA/OMA and double PRISs-NOMA in delay-limited and delay-tolerant modes.
        In delay-limited mode, the system throughput of PRIS-ARIS-NOMA outperforms double PRISs-NOMA while that of double PRISs-NOMA is superior to PRIS-ARIS-OMA.
        In delay-tolerant mode, the system throughput of PRIS-ARIS-NOMA with pSIC outperforms that of PRIS-ARIS-NOMA with ipSIC.
\end{enumerate}
\subsection{Organizations and Notations}
The remainder of this paper is structured as follows.
In Section II, system model of PRIS-ARIS-NOMA is introduced.
In Section III, the outage behaviors of PRIS-ARIS-NOMA are analyzed meticulously.
In Section IV, the ergodic data rates of the two users are derived in detail.
In Section V, simulation results and analysis are presented to consolidate the conclusions summarized records in Section VI.

Notations in this paper describe mainly as follows: ${f_X}\left(  \cdot  \right)$ and ${F_X}\left(  \cdot  \right)$ are denote the the
probability density function (PDF) and the cumulative distribution function (CDF) of a random variable $X$, respectively; ${\mathbb{E[ \cdot ]}}$ denotes the expectation operator; $diag\left(  \cdot  \right)$ and ${\left(  \cdot  \right)^H}$ represent the diagonal matrix and conjugate-transpose operations, respectively.
\section{System Model}\label{System Model}
We consider double RISs, i.e., PRIS and ARIS assisted downlink NOMA communication scenarios, shown in Fig.1.
Compared with the multiuser NOMA scenario, the use of two-user NOMA model in this paper can more intuitively show the performance comparison between the nearby and distant users, and better analyze the fairness and superiority of the NOMA communication system compared with the orthogonal multiple access system in the case of large differences in the channel conditions compared with the multiuser scenario\footnote{In this paper we focus on two-user scenario and multi-user scenario will be analyzed in future studies.}.
Then we set the the nearby and distant user which are denoted as ${D_n}$ and ${D_m}$, respectively.
The user is equipped with a receiver and has direct access to its own signal.
The PRIS is deployed near the base station (BS) while the ARIS is arranged around the users.
An amplify-forward (AF) relay positioned between the PRIS and ARIS consolidates signals reflected from the PRIS and reliably transmits them to the ARIS.
And it can also be integrated onto drones or aerial platforms to facilitate longer-range signal transmission in practical applications \cite{double_RIS_UAV}.
Studies in \cite{R13} have shown that RIS and relay can be used together to extend the coverage and improve the performance of the networks with an appropriate deployment strategy.
Also the advantage of using AF relay over decode-forward (DF) relay is that it has a higher transmission rate and wider extension range with less complexity and is more suitable for multi-hop relaying networks to reduce the computational complexity \cite{R14}.
AF relay does not require a complex network of reflecting elements to regulate the propagation of signals, but only amplifies and forwards the signals, which makes the implementation of AF relay relatively simple and reduces the complexity, and also minimizes the difficulty of phase-shift calculations brought about by the RIS when mathematically modelling the communication model.
RIS can be adjusted to control the process of transmitting and receiving in real time to the maximum extent possible according to the environment and communication needs, so as to achieve the regulation and optimisation of wireless signals.
Therefore we chose to deploy RIS at the transmitting end and receiving end, and AF relay at the intermediate node.
Considering that deploying ARIS at the transmitting end may bring more noise interference problems, we choose to deploy PRIS at the transmitting end and ARIS at the receiving end.
Specifically, the ARIS and PRIS are equipped with $M$ and $N$ reflecting elements (REs), respectively.
In this paper we focus on analyzing the system performance of assisted communication via RISs in scenarios where the direct link is blocked, thus there is no direct communication link between BS and the users.
The BS and each users are equipped with only one single antenna\footnote{Multi-antenna scenario is a totally different scenario and will be exploited in the future work.}.
If BS is equipped with multiple antennas, the user's reachable rate will be increased by a multiple of the number of antennas from a single antenna, therefore it is modelled as a single antenna BS in this model to analyze users performance more intuitively.
In this model, two-RIS links are constructed with AF relay as the demarcation.
The first part channels of the link from BS to PRIS then to AF relay can be defined as ${{\mathbf{h}}_1} \in {\mathbb{C}^{N \times 1}}$ and ${{\mathbf{h}}_2} \in {\mathbb{C}^{N \times 1}}$, respectively.
The second part channels of the link from AF relay to ARIS then to users can be defined as ${{\mathbf{g}}_1} \in {\mathbb{C}^{M \times 1}}$ and ${{\mathbf{g}}_\varphi } \in {\mathbb{C}^{M \times 1}}$, respectively, where $\varphi  \in \left\{ {m,n} \right\}$.
Define the phase shifting matrix of ARIS as ${{\mathbf{\Theta }}_a} = diag\left( {\sqrt {\beta _1^a} {e^{j{\theta _1}}},\sqrt {\beta _2^a} {e^{j{\theta _2}}}, \ldots ,\sqrt {\beta _M^a} {e^{j{\theta _M}}}} \right)$,
and define the phase shifting matrix of PRIS as ${{\mathbf{\Theta }}_p} = diag\left( {\sqrt {\beta _1^p} {e^{j{\varphi _1}}},\sqrt {\beta _2^p} {e^{j{\varphi _2}}}, \ldots ,\sqrt {\beta _N^p} {e^{j{\varphi _N}}}} \right)$, where ${{\theta _m}}$ and ${{\varphi _n}}$ denote the $m$-th and $n$-th element of the ARIS and PRIS, respectively.
$\beta$ indicates the reflection amplification of the signals by RISs.
Compared to PRIS, because that ARIS is powered by an external power source, it can amplify and forward signals that could just be normally reflected by the PRIS.
For simplicity, all ${\beta _n^p}$ are assumed to be $\beta _n^p = 1$ while all $\beta _m^a$ are assumed to be ${\beta} = \beta _m^a > 1$.
Then the phase shifting matrix of ARIS and PRIS can be denoted as ${{\mathbf{\Theta }}_a} = \sqrt \beta  diag\left( {{e^{j{\theta _1}}},{e^{j{\theta _2}}}, \ldots ,{e^{j{\theta _M}}}} \right) = \sqrt \beta  {{\mathbf{\Phi }}_a}$, and ${{\mathbf{\Theta }}_p} = diag\left( {{e^{j{\varphi _1}}},{e^{j{\varphi _2}}}, \ldots ,{e^{j{\varphi _N}}}} \right) = {{\mathbf{\Phi }}_p}$, respectively.
\begin{figure}[t!]
    \begin{center}
        \includegraphics[width=3.4in,  height=1.8in]{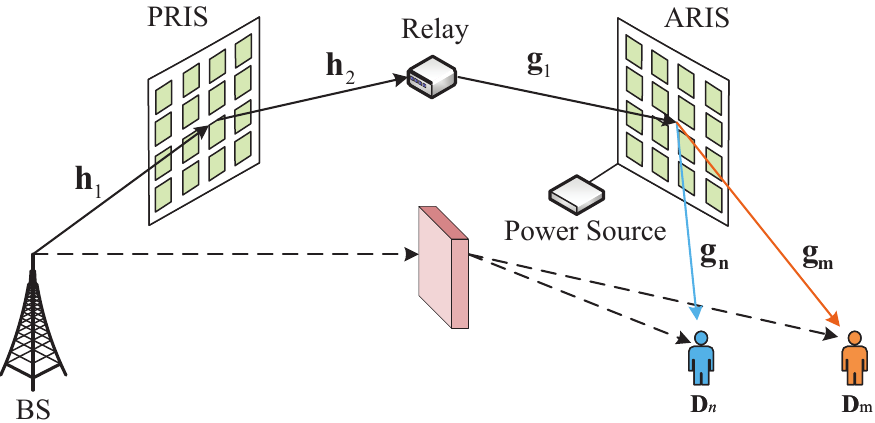}
        \caption{System model of PRIS-ARIS-NOMA.}
        \label{Fig. 1}
    \end{center}
\end{figure}
\subsection{Signal Model}
Although ARIS can use an external power source to combat the channel fading of signal propagation, it also leads to the amplification of noise signals on it, which cannot be idealized as in the analysis of a single PRIS model and the amplified noise from the transmitter on ARIS cannot be ignored in the analysis process.
Hence, the expression of the received signal when it passes through the whole channel link and then is received at the $\varphi$-th user is denoted as
\begin{align}\label{received signal expression}
  {y_\varphi } = \sqrt \beta  {\mathbf{g}}_\varphi ^H{{\mathbf{\Phi }}_a}&\left[ {{\mathbf{h}}_2^H{{\mathbf{\Phi \nonumber }}_p}{{\mathbf{h}}_1}\left( {\sqrt {{a_n}{P_s^{a}}} {x_n}} \right.} \right. \hfill \nonumber \\
  &\left. {\left. { + \sqrt {{a_m}{P_s}} {x_m}} \right)G{{\mathbf{g}}_1} + {{\mathbf{N}}_a}} \right] + {n_0},
\end{align}
where ${{\mathbf{\Phi }}_a}$ and ${{\mathbf{\Phi }}_p}$ represent the reflecting phase shift matrices of ARIS and PRIS, which have been given a detailed explanation and definition above.
${{\mathbf{N}}_a}\sim\mathcal{C}\mathcal{N}\left( {0,\sigma _a^2{{\mathbf{I}}_N}} \right)$ represents the thermal noise, and ${\mathbf{I}}_N \in {\mathbb{C}^{M \times 1}}$ is a unit column vector.
$a_n$ and $a_m$ are power allocation factor of $D_n$ and $D_m$, respectively, with $a_n+a_m=1$.
$P_s^{a}$ is the transmitting power of the BS for PRIS-ARIS-NOMA with $x_n$ and $x_m$ are normalized power signals.
For elaboration, $\mathbb{E}\left\{ {{{\left| {{x_n}} \right|}^2}} \right\} = \mathbb{E}\left\{ {{{\left| {{x_m}} \right|}^2}} \right\} = 1$.
Meanwhile, ${n_0}\sim\mathcal{C}\mathcal{N}(0,{\sigma ^2})$ is the additive white Gaussian noise (AWGN) at the receiving end, ${\sigma ^2}$ is the power of AWGN.
$G$ is the AF amplification gain and set $G$ to 1 for calculations.

Rice fading channels and Nakagami-$m$ fading channels are the more commonly used channel modelling approaches in performance analysis scenarios and are commonly used to analyze and model signal transmission performance in wireless communication systems.
The fading characteristics of the Rician fading channel consist mainly of fading from the dominant path, but also from the multipath component, and are therefore suitable for modelling the transmission channel between RISs in this model.
The Nakagami-$m$ fading channel model is applicable to a variety of different wireless environments, including indoor, outdoor, urban and suburban areas, and is therefore well suited for the transmitting link as well as the receiving link in this model.
In order to better fit the actual channel characteristics, we considered the channels from BS to PRIS ,then from ARIS to users as Nakagami-$m$ fading channels.
Assuming the channels between the two RISs are Rician fading channels.
Based on the above definition of the communication channels, we can denote the channel coefficients from BS to PRIS and PRIS to AF as ${{\bf{h}}_1} = \sqrt {\eta d_{h1}^{ - \alpha }} \left[ {h_1^1,h_2^1 \ldots ,h_N^1} \right]^H$ and ${{\bf{h}}_2} = \sqrt {\eta d_{h2}^{ - \alpha }} \left[ {h_1^2,h_2^2 \ldots ,h_N^2} \right]^2$, respectively.
And meanwhile denote the channel coefficients from AF to ARIS and ARIS to the $\varphi$-user as ${{\bf{g}}_1} = \sqrt {\eta d_{g1}^{ - \alpha }} {\left[ {g_1^1,g_2^1 \ldots ,g_M^1} \right]^H}$ and ${{\bf{g}}_\varphi } = \sqrt {\eta d_{g\varphi }^{ - \alpha }} {\left[ {g_1^\varphi ,g_2^\varphi  \ldots ,g_M^\varphi } \right]^H}$, respectively.
$\eta $ expresses the the pass loss while $\alpha$ expresses the path loss exponent.
$d_{h1}$,$d_{h2}$,$d_{g1}$,$d_{g\varphi}$ respectively denote the distance of each channel segment.
$h_n^1$ and $g_m^\varphi$ follows the Nakagami-$m$ fading with shape parameter $m_{na}$ and scale parameter $\Omega_{na}$.
$h_n^2 = \sqrt {\frac{\kappa }{{\kappa  + 1}}}  + \sqrt {\frac{1}{{\kappa  + 1}}} \widehat {h_n^2}$ and $g_m^1 = \sqrt {\frac{\kappa }{{\kappa  + 1}}}  + \sqrt {\frac{1}{{\kappa  + 1}}} \widehat {g_m^1}$, where $\widehat {g_m^1}$ and $\widehat {h_n^2}\sim{\mathcal{C}}{\mathcal{N}}\left( {0,1} \right)$.
$\kappa$ is the Rician factor, which represents the ratio of the LoS component power to the scattered component power.
In order to make the analysis more accurate, all of the complex channel coefficients random variables which are from each part of channels are independent, and each part of channels is assumed to follow independent block fading, either.
It is to be noted that the normalized signals power is taken in the following calculation.

On the basis of the NOMA protocol, the transmitting end sends the superposed signal to each pair of non-orthogonal users, and the receiving end starts SIC procedure for decoding the signals.
The received signal-to-interference-plus-noise ratio (SINR) at $D_n$ to detected $D_m$'s information $x_m$ can be denoted by
\begin{align}\label{SINRND}
\gamma _{{D_n}}^{{D_m}} = \frac{{\beta {P_s^{a}}{{\left| {{\bf{h}}_2^H{{\bf{\Phi }}_p}{{\bf{h}}_1}} \right|}^2}{{\left| {{\bf{g}}_n^H{{\bf{\Phi }}_a}{{\bf{g}}_1}} \right|}^2}{a_m}}}{{\beta {P_s^{a}}{{\left| {{\bf{h}}_2^H{{\bf{\Phi }}_p}{{\bf{h}}_1}} \right|}^2}{{\left| {{\bf{g}}_n^H{{\bf{\Phi }}_a}{{\bf{g}}_1}} \right|}^2}{a_n} + \beta \sigma _a^2{{\left\| {{\bf{g}}_n^H{{\bf{\Phi }}_a}} \right\|}^2} + {\sigma ^2}}}.
\end{align}

After $D_n$ decoded and subtracted $D_m$'s signals $x_m$, it would next decode its own signals $x_n$. The received SINR at $D_n$ to detect its own information $x_n$ can be denoted by
\begin{align}\label{SINRN}
{\gamma _{{D_n}}} = \frac{{\beta {P_s^{a}}{{\left| {{\bf{h}}_2^H{{\bf{\Phi }}_p}{{\bf{h}}_1}} \right|}^2}{{\left| {{\bf{g}}_n^H{{\bf{\Phi }}_a}{{\bf{g}}_1}} \right|}^2}{a_n}}}{{\varpi {P_s^{a}}{{\left| {{h_{RI}}} \right|}^2} + \beta \sigma _a^2{{\left\| {{\bf{g}}_n^H{{\bf{\Phi }}_a}} \right\|}^2} + {\sigma ^2}}},
\end{align}
where $\varpi  = {\text{0 or 1}}$ represents the pSIC and ipSIC operation during the SIC process, respectively. ${h_{RI}}\sim\mathcal{C}\mathcal{N}\left( {0,{\Omega _{RI}}} \right)$ represents the corresponding complex coefficient of the residual interference during the ipSIC operation.
Noted that the generation of ipSIC is mainly due to the imperfect design of the SIC receiver as well as hardware errors, whereas in this paper we use the mathematical modelling of stochastic processes to model the mathematical properties of the whole channel, so that the CSI case received by the users can be assumed to be perfect\footnote{Imperfect CSI estimation in RIS networks is not considered in this paper and is the focus of our future researches.}.

When it comes to $D_m$'s demodulation part, the information belonged to $D_n$ is treated as interference during the SIC process. Therefore, the received SINR at $D_m$ to detected its own information $x_m$ can be denoted by
\begin{align}\label{SINRD}
{\gamma _{{D_m}}} = \frac{{\beta {P_s^{a}}{{\left| {{\bf{h}}_2^H{{\bf{\Phi }}_p}{{\bf{h}}_1}} \right|}^2}{{\left| {{\bf{g}}_m^H{{\bf{\Phi }}_a}{{\bf{g}}_1}} \right|}^2}{a_m}}}{{\beta {P_s^{a}}{{\left| {{\bf{h}}_2^H{{\bf{\Phi }}_p}{{\bf{h}}_1}} \right|}^2}{{\left| {{\bf{g}}_m^H{{\bf{\Phi }}_a}{{\bf{g}}_1}} \right|}^2}{a_n} + \beta \sigma _a^2{{\left\| {{\bf{g}}_m^H{{\bf{\Phi }}_a}} \right\|}^2} + {\sigma ^2}}}.
\end{align}

In order to provide a comparative basis for one of the experimental results, OMA scheme was presented as a baseline for comparison with NOMA scheme. The SNR at the nearby user $D_n$ and distant user $D_m$ under OMA scheme can be given by
\begin{align}\label{SNROMAN}
\gamma _{{D_n}}^{OMA} = \frac{{\beta {P_s}{a_n}{{\left| {{\bf{h}}_2^H{{\bf{\Phi }}_p}{{\bf{h}}_1}} \right|}^2}{{\left| {{\bf{g}}_n^H{{\bf{\Phi }}_a}{{\bf{g}}_1}} \right|}^2}}}{{\beta \sigma _a^2{{\left\| {{\bf{g}}_n^H{{\bf{\Phi }}_a}} \right\|}^2} + {\sigma ^2}}}
\end{align}
and
\begin{align}\label{SNROMAD}
\gamma _{{D_m}}^{OMA} = \frac{{\beta {P_s}{a_m}{{\left| {{\bf{h}}_2^H{{\bf{\Phi }}_p}{{\bf{h}}_1}} \right|}^2}{{\left| {{\bf{g}}_m^H{{\bf{\Phi }}_a}{{\bf{g}}_1}} \right|}^2}}}{{\beta \sigma _a^2{{\left\| {{\bf{g}}_m^H{{\bf{\Phi }}_a}} \right\|}^2} + {\sigma ^2}}},
\end{align}
respectively.
\subsection{Channel Statistical Properties}
In order to further improve the performance the networks, coherent phase shifting scheme has been adopted to optimize the phase shifting process \cite{Ding_phase_shifting}.
In a coherent phase shifting scheme, the phase shifts of each reflecting and transmitting element are synchronized with the phases of their respective incoming and outgoing fading channels.
This coherence results in superior analytical performance compared to the random phase shifting scheme.
It can simulate a more idealized and better channel state, and can reflect the system characteristics more intuitively and conveniently.
In coherent phase shifting scheme, when the distance is normalized, we can define the cascade channel gain ${\left| {{\bf{h}}_2^H{\bf{\Theta }}{{\bf{h}}_1}} \right|^2}$ as ${\zeta _h} = {\left| {\sum\nolimits_{n = 1}^N {\left| {h_n^1h_n^2} \right|} } \right|^2}$, and define ${\left| {{\bf{g}}_\varphi ^H{\bf{\Theta }}{{\bf{g}}_1}} \right|^2}$ as ${\zeta _g} = {\left| {\sum\nolimits_{m = 1}^M {\left| {g_m^1g_m^\varphi } \right|} } \right|^2}$.
Let ${X_n} = \left| {h_n^1h_n^2} \right|$ and $X_m^\varphi  = \left| {g_m^1g_m^\varphi } \right|$, thus we can get the PDF for a single cascaded Rician fading and Nakagami-$m$ fading channel as
\begin{align}\label{Single Cascade Channel PDF}
  {f_{{X_n}}}\left( y \right) = {f_{{X_m^{\varphi}}}}\left( y \right) =& \sum\limits_{m = 0}^\infty  {\frac{{4{y^{m + {m_{na}}}}{\kappa ^m}{{\left[ {\frac{{\left( {\kappa  + 1} \right){m_{na}}}}{{\hat r_1^2{{\hat r}_2}}}} \right]}^{\frac{{1 + m + {m_{na}}}}{2}}}}}{{{e^\kappa }{{\left( {m!} \right)}^2}\Gamma \left( {{m_{na}}} \right)}}}  \hfill \nonumber \\
   &\times {K_{m - {m_{na}} + 1}}\left( {2y\sqrt {\frac{{\left( {\kappa  + 1} \right){m_{na}}}}{{\hat r_1^2{{\hat r}_2}}}} } \right),
\end{align}
where ${\hat r_i}$ is the root-mean square value of the received signal envelope ${R_i}$ \cite{kuandnudistribution},
$\Gamma ( \cdot )$ is the gamma function\cite[Eq. (8.310.1)]{table}. ${K_\phi }( \cdot )$ is the $\phi$th-order modified Bessel function of the second kind\cite[Eq. (8.432)]{table}.

Since all channel parameters are distributed independently, $X_n$ and $X_m$ have the same mean and variance, which can be given by
\begin{align}\label{Mean}
{e_\varphi } = \sqrt {\frac{\pi }{{4\left( {\kappa  + 1} \right)}}} \frac{{\Gamma \left( {{m_{na}} + \frac{1}{2}} \right)}}{{\Gamma \left( {{m_{na}}} \right)}}{\left( {\frac{1}{{{m_{na}}}}} \right)^{\frac{1}{2}}}{L_{\frac{1}{2}}}\left( { - \kappa } \right),
\end{align}
and
\begin{align}\label{Variance}
{d_\varphi } = 1 - \frac{{L_{\frac{1}{2}}^2\left( { - \kappa } \right)}}{{4\left( {\kappa  + 1} \right)}}\left[ {\frac{\pi }{m}{{\left( {\frac{{\Gamma \left( {m + \frac{1}{2}} \right)}}{{\Gamma \left( m \right)}}} \right)}^2}} \right],
\end{align}
respectively, where ${L_{\frac{1}{2}}}\left( x \right)$ is the Laguerre polynomial and ${L_{\frac{1}{2}}}\left( { - \kappa } \right) = {e^{ - \frac{\kappa }{2}}}\left[ {\left( {1 + \kappa } \right){I_0}\left( {\frac{\kappa }{2}} \right) + \kappa {I_1}\left( {\frac{\kappa }{2}} \right)} \right]$, where ${I_0}\left(  \cdot  \right)$ and ${I_1}\left(  \cdot  \right)$ are respectively the modified zeroth-order and first-order Bessel function of the first kind.

Then, with the help of Laguerre polynomial series\cite[Sec. 2.2.2]{Stochastic_Methods}, the approximate PDF of ${\zeta _h}$ and ${\zeta _g}$ can be given by
\begin{align}\label{zetahpdf}
{f_{{\zeta _h}}}\left( x \right) = \frac{{{x^{\frac{{{a_X} - 1}}{2}}}}}{{2{b^{{a_X} + 1}}\Gamma \left( {{a_X} + 1} \right)}}{e^{ - \frac{{\sqrt x }}{{{b_X}}}}},
\end{align}
and
\begin{align}\label{zetagpdf}
{f_{{\zeta _g}}}\left( y \right) = \frac{{{y^{\frac{{{a_Y} - 1}}{2}}}}}{{2{b^{{a_Y} + 1}}\Gamma \left( {{a_Y} + 1} \right)}}{e^{ - \frac{{\sqrt y }}{{{b_Y}}}}},
\end{align}
respectively, where ${a_X} = \frac{{Ne_\varphi ^2}}{{{d_\varphi }}} - 1$, ${a_Y} = \frac{{Me_\varphi ^2}}{{{d_\varphi }}} - 1$,${b_X} = {b_Y} = \frac{{{d_\varphi }}}{{{e_\varphi }}}$.
\section{Outage Probability}\label{Outage probability}
This section evaluates the performance of the PRIS-ARIS-NOMA in the perspective of outage probability.
In this section the outage probability expressions for $D_n$ with ipSIC/pSIC and $D_m$ are derived in detail.
And the outage probability expressions of $D_n$ with ipSIC/pSIC and $D_m$ at high SNR region are further derived based on the outage probability expressions.

\subsection{The Outage Probability of $D_n$}
According to the regulations of SIC process, the weak user needs to decode and discard the signals of the strong user, i.e., the signals of $D_m$, before decoding its own signals in the SIC process.
Therefore, the event that causes an outage to $D_n$ can be defined as:1) If $D_n$ cannot detect the signals $x_m$ from $D_m$, then an outage occurs at $D_n$.
2) If $D_n$ can detect and decode $D_m$'s signal $x_m$, but cannot detect and decode its own signal $x_n$, then an outage occurs at $D_n$.
Based on the above representations, the outage probability of $D_n$ in PRIS-ARIS-NOMA can be denoted as
\begin{align}\label{outage event Dn}
P_{{D_n}}^{out} = \Pr \left( {\gamma _{{D_n}}^{{D_m}} < {\gamma _{t{h_m}}}} \right) + \Pr \left( {\gamma _{{D_n}}^{{D_m}} > {\gamma _{t{h_m}}},{\gamma _{{D_n}}} < {\gamma _{t{h_n}}}} \right),
\end{align}
where ${\gamma _{t{h_n}}} = {2^{{R_n}}} - 1$ and ${\gamma _{t{h_m}}} = {2^{{R_m}}} - 1$ are respectively the target SNR threshold when detecting signals. And $R_n$ and $R_m$ are respectively the instantaneous target rate threshold for $D_n$ and $D_m$.
\begin{theorem} \label{theorem:1}
Under cascaded Rician fading and Nakagami-$m$ fading channels, the closed-form expression for outage probability of $D_n$ with ipSIC for PRIS-ARIS-NOMA can be expressed as
\begin{align}\label{OP for Dn with ipSIC}
  &P_{{D_n}}^{out,ipSIC} = \sum\limits_{i = 1}^I {\sum\limits_{u = 1}^U {\frac{{{H_i}{H_u}x_u^{{a_X}}}}{{\Gamma ({a_X} + 1)\Gamma ({a_Y} + 1)}}} }  \hfill \nonumber \\
   &\times \gamma \left( {{a_Y} + 1,\frac{{\sqrt {\overline {{\lambda _n}} {\gamma _{t{h_n}}}\left( {\varpi {P_s^{a}}{\Omega _{RI}}{x_i} + \beta \sigma _a^2M{\Omega _{na}} + {\sigma ^2}} \right)} }}{{\sqrt {\beta {P_s^{a}}{a_n}} {b_X}{b_Y}{x_u}}}} \right),
\end{align}
where $\varpi  = 1$, $\overline {{\lambda _n}}  = {\eta ^{ - 4}}d_{h1}^\alpha d_{h2}^\alpha d_{g1}^\alpha d_{gn}^\alpha $.
For bringing into Gauss-Laguerre quadrature, ${{x_u}}$ is the $u$-th zero point of Laguerre polynomial ${L_u}({x_u})$ and ${{x_i}}$ is the $i$-th zero point of Laguerre polynomial ${L_i}({x_i})$. ${{H_u}}$ is the $u$-th weight, and can be denoted by ${H_u} = \frac{{{{[(U + 1)!]}^2}}}{{{x_u}{{[{{L'}_{U + 1}}({x_u})]}^2}}},u = 0,1,...,U$. Similarly, ${{H_i}}$ get the same mathematical meaning as ${{H_u}}$ and can be denoted by ${H_i} = \frac{{{{[(I + 1)!]}^2}}}{{{x_i}{{[{{L'}_{I + 1}}({x_i})]}^2}}},i = 0,1,...,I$.
Specifically, ${L_\varphi}\left( x \right) = {e^x}\frac{{{d^\varphi}}}{{d{x^\varphi}}}\left( {{x^\varphi}{e^{ - x}}} \right)$ is $\varphi$-th order Laguerre polynomial and ${L_\varphi}^\prime \left( x \right)$ is the derivative of ${L_\varphi} \left( x \right)$.
${\gamma (\alpha ,x) = \int_0^x {{e^{ - t}}{t^{\alpha  - 1}}} dt}$ is the lower incomplete Gamma function\cite[Eq. (8.350.1)]{table}.
Particularly, formula is valid under the condition of ${{a_m} > {\gamma _{t{h_m}}}{a_n}}$, if ${{a_m} < {\gamma _{t{h_m}}}{a_n}}$, then the nearby user will always be in the outage state.

\begin{proof} See Appendix A.
\end{proof}
\end{theorem}
\begin{corollary} \label{corollary1}
When it comes to the special case of $\varpi  = 0$, the closed-formed expression for outage probability of $D_n$ with pSIC for PRIS-ARIS-NOMA can be expressed as
\begin{align}\label{OP for Dn with pSIC}
  P_{{D_n}}^{out,pSIC} =& \sum\limits_{u = 1}^U {\frac{{{H_u}x_u^{{a_X}}}}{{\Gamma ({a_X} + 1)\Gamma ({a_Y} + 1)}}}  \hfill \nonumber \\
   &\times \gamma \left( {{a_Y} + 1,\frac{{\sqrt {\overline {{\lambda _n}} {\gamma _{t{h_n}}}\left( {\beta \sigma _a^2M{\Omega _{na}} + {\sigma ^2}} \right)} }}{{\sqrt {\beta P_s^a{a_n}} {b_X}{b_Y}{x_u}}}} \right).
\end{align}
\end{corollary}

\subsection{The Outage Probability of $D_m$}
In NOMA scenario, for distant user during the SIC process, only its own signals need to be decoded and the signals of the nearby user are treated as interference.
Therefore, the event that causes an outage to $D_m$ can be defined as: If $D_m$ cannot detect the signals $x_m$ itself, then an outage occurs at $D_m$.
Based on the above representations, the outage probability of $D_m$ in PRIS-ARIS-NOMA can be denoted as
\begin{align}\label{outage event Dm}
P_{{D_m}}^{out} = \Pr \left( {{\gamma _{{D_m}}} < {\gamma _{t{h_m}}}} \right).
\end{align}
\begin{theorem} \label{theorem:2}
Under cascaded Rician fading and Nakagami-$m$ fading channels, the closed-form expression for outage probability of $D_m$ for PRIS-ARIS-NOMA can be expressed as
\begin{align}\label{OP for Dm}
  P_{{D_m}}^{out} =& \sum\limits_{u = 1}^U {\frac{{{H_u}x_u^{{a_X}}}}{{\Gamma ({a_X} + 1)\Gamma ({a_Y} + 1)}}}  \hfill \nonumber \\
   &\times \gamma \left( {{a_Y} + 1,\frac{{\sqrt {\overline {{\lambda _m}} {\gamma _{t{h_m}}}\left( {\beta \sigma _a^2M{\Omega _{na}} + {\sigma ^2}} \right)} }}{{\sqrt {\beta {P_s^{a}}\left( {{a_m} - {\gamma _{t{h_m}}}{a_n}} \right)} {b_X}{b_Y}{x_u}}}} \right),
\end{align}
where $\overline {{\lambda _m}}  = {\eta ^{ - 4}}d_{h1}^\alpha d_{h2}^\alpha d_{g1}^\alpha d_{gm}^\alpha $, ${{a_m} > {\gamma _{t{h_m}}}{a_n}}$.
\begin{proof}
By substituting \eqref{SINRD} into \eqref{outage event Dm}, the outage probability expression of $D_m$ can be further derived as \eqref{OP for Dm step1} at the top of the next page.
\begin{figure*}[!t]
\normalsize
\begin{align}\label{OP for Dm step1}
  P_{{D_m}}^{out} &= \Pr \left( {\frac{{\beta {P_s^{a}}{{\left| {{\mathbf{h}}_2^T{{\mathbf{\Phi }}_p}{{\mathbf{h}}_1}} \right|}^2}{{\left| {{\mathbf{g}}_m^T{{\mathbf{\Phi }}_a}{{\mathbf{g}}_1}} \right|}^2}{a_m}}}{{\beta {P_s^{a}}{{\left| {{\mathbf{h}}_2^T{{\mathbf{\Phi }}_p}{{\mathbf{h}}_1}} \right|}^2}{{\left| {{\mathbf{g}}_m^T{{\mathbf{\Phi }}_a}{{\mathbf{g}}_1}} \right|}^2}{a_n} + \beta \sigma _a^2{{\left\| {{\mathbf{g}}_n^T{{\mathbf{\Phi }}_a}} \right\|}^2} + {\sigma ^2}}} < {\gamma _{t{h_m}}}} \right).
\end{align}
\hrulefill \vspace*{0pt}
\end{figure*}
The remaining computational steps of the proof are similar to the proof steps in Appendix A.
The proof is completed.
\end{proof}
\end{theorem}
\subsection{The Outage Probability of the OMA Benchmark}
In this paper, the OMA scheme is set up as one of the benchmarks for comparison with the NOMA scheme under the double RISs assisted networks.
For OMA scheme, the transmission of users'informations throughout the system is orthogonal in the frequency domain, and in a dual users system, it can be assumed that each of the two users occupies half of the frequency domain resources. For OMA scheme, an outage occurs when the user's received SNR is less than the set SNR threshold, which can be expressed as
\begin{align}\label{OP for OMA definition}
P_{{D_\varphi }}^{OMA} = \Pr \left( {\gamma _{{D_\varphi }}^{OMA} < \gamma _{t{h_\varphi }}^{OMA}} \right),
\end{align}
where $\gamma _{t{h_\varphi }}^{OMA} = {2^{2{R_\varphi }}} - 1$ represents the target SNR threshold when detecting the $\varphi$-th user's signals, and $R_{\varphi}$ is the instantaneous target rate threshold for the $\varphi$-th user.
Similar to the above derivation process, the outage probabilities for the two users under double RISs assisted OMA networks are presented in the following theorem.
\begin{theorem} \label{theorem:3}
Under Rician fading and Nakagami-$m$ fading cascaded channels, the approximated closed-form expression for outage probability of the $\varphi$-th user for double RISs assisted OMA networks can be expressed as
\begin{small}
\begin{align}\label{OP for OMAbenchmark}
  P_{{D_\varphi }}^{OMA} =& \sum\limits_{u = 1}^U {\frac{{{H_u}x_u^{{a_X}}}}{{\Gamma ({a_X} + 1)\Gamma ({a_Y} + 1)}}}  \hfill \nonumber \\
   &\times \gamma \left( {{a_Y} + 1,\frac{{\sqrt {\overline {{\lambda _\varphi }} \gamma _{t{h_\varphi }}^{OMA}\left( {\beta \sigma _a^2M{\Omega _{na}} + {\sigma ^2}} \right)} }}{{\sqrt {\beta {P_s^{a}}{a_\varphi }} {b_X}{b_Y}{x_u}}}} \right),
\end{align}
\end{small}where $\varphi  \in \left\{ {m,n} \right\}$.
\end{theorem}
\subsection{Diversity Analysis}
Diversity order is an important metric for measuring system communication performance and describes how fast the outage probability decreases with an increasing SNR \cite{div}.
The diversity order can be expressed as the ratio of the asymptotic outage probability in the high SNR region, which can be expressed as
\begin{align}\label{Definition of the diversity order}
\mathop {\lim }\limits_{{P_s^{a}} \to \infty }  - \frac{{\log P_{out}^\infty \left( {{P_s^{a}}} \right)}}{{\log {P_s^{a}}}},
\end{align}
where ${P_{out}^\infty \left( {{P_s^{a}}} \right)}$ represents the asymptotic outage probability at high SNR region, .i.e, ${{P_s^{a}} \to \infty }$.
The asymptotic outage probability of $D_n$ with ipSIC can be directly deduced from \eqref{OP for Dn with ipSIC} and leads to the following corollary.
\begin{corollary} \label{corollary2}
When ${{P_s^{a}} \to \infty }$, the expression for asymptotic outage probability of $D_n$ with ipSIC for PRIS-ARIS-NOMA under cascaded Rician fading and Nakagami-$m$ fading channels can be expressed as
\begin{align}\label{asymptotic outage probability of Dn with ipSIC}
  P_{out,{D_n}}^{\infty ,ipSIC} =& \sum\limits_{i = 1}^I {\sum\limits_{u = 1}^U {\frac{{{H_i}{H_u}x_u^{{a_X}}}}{{\Gamma ({a_X} + 1)\Gamma ({a_Y} + 1)}}} }  \hfill \nonumber \\
   &\times \gamma \left( {{a_Y} + 1,\frac{{\sqrt {\overline {{\lambda _n}} {\gamma _{t{h_n}}}{\Omega _{RI}}{x_i}} }}{{\sqrt {\beta {a_n}} {b_X}{b_Y}{x_u}}}} \right).
\end{align}
\end{corollary}
\begin{remark} \label{remark1}
From \textbf{Corollary\ref{corollary2}}, it can be deduced that the outage probability of $D_n$ with ipSIC will gradually converge to a constant as SNR tends to infinity. By substituting \eqref{asymptotic outage probability of Dn with ipSIC} into \eqref{Definition of the diversity order}, a zero diversity order for $D_n$ with ipSIC can be derived.
The reason for this phenomenon is that the residual interference during the ipSIC process limits the user acceptance performance improvement as the SNR increases, thus causing the outage performance to converge to a constant state.
\end{remark}
\begin{corollary} \label{corollary3}
With the condition of $N=M$, when ${{P_s^{a}} \to \infty }$, the expression for asymptotic outage probability of $D_n$ with pSIC for PRIS-ARIS-NOMA under cascaded Rician fading and Nakagami-$m$ fading channels can be expressed as \eqref{asymptotic outage probability of Dn with pSIC} at the top of the next page.
\begin{figure*}[!t]
\normalsize
\begin{small}
\begin{align}\label{asymptotic outage probability of Dn with pSIC}
P_{out,{D_n}}^{\infty ,pSIC} = \frac{{\pi {\mu _b}{{\left( {\overline {{\lambda _n}} {\gamma _{t{h_n}}}} \right)}^M}{{\left( {\beta \sigma _a^2M{\Omega _{na}} + {\sigma ^2}} \right)}^M}}}{{2K\left( {2M} \right)!\left( {2N - 1} \right)!{{\left( {\beta {P_s}{a_n}} \right)}^M}}}{\left[ {\frac{{4\sqrt \pi  \Gamma \left( {2{m_{na}}} \right)\Delta \left( 0 \right)}}{{{e^\kappa }\Gamma \left( {{m_{na}}} \right)}}} \right]^{M + N}}\sum\limits_{k = 1}^K {{{\left[ {\frac{{\left( {{x_k}{\text{ + }}1} \right){\mu _b}}}{2}} \right]}^{2N - 2M - 1}}\sqrt {1 - {x_k^2}} }.
\end{align}
\end{small}
\hrulefill \vspace*{0pt}
\end{figure*}
Where $K$ is the parameter of the Gauss-Chebyshev quadrature formula for precision.
$\mu _b$ is the upper limit of integral before using the Gauss-Chebyshev quadrature formula for the approximation and ${\mu _b} \to \infty $. ${x_k} = \cos \left( {\frac{{2k - 1}}{{2K}}\pi } \right)$.
$\Delta \left( 0 \right)$ is given as following where $F\left( { \cdot , \cdot ; \cdot ; \cdot } \right)$ is the Gauss hypergeometric function\cite[Eq. (9.100)]{table}.
\begin{align}\label{deta0}
  \Delta \left( 0 \right) =& \frac{{{{\left( {4\sqrt {\left( {\kappa  + 1} \right){m_{na}}} } \right)}^{1 - {m_{na}}}}{{\left[ {\left( {\kappa  + 1} \right){m_{na}}} \right]}^{\frac{1}{2}\left( {1 + {m_{na}}} \right)}}}}{{\Gamma \left( {{m_{na}} + \frac{3}{2}} \right)}} \hfill \nonumber \\
   &\times F\left( {2,\frac{3}{2} - {m_{na}};{m_{na}} + \frac{3}{2};1} \right)
\end{align}
\end{corollary}
\begin{proof} See Appendix B.
\end{proof}
\begin{remark} \label{remark2}
From \textbf{Corollary \ref{corollary3}}, it can be deduced that the asymptotic outage probability of $D_n$ with pSIC is an oblique line that decreases as SNR increases.
By substituting \eqref{asymptotic outage probability of Dn with pSIC} into \eqref{Definition of the diversity order}, the diversity order of $D_n$ with pSIC can be derived as $M$, which is related to the number of the ARIS's reflecting elements.
\end{remark}
\begin{corollary} \label{corollary4}
In the same analogy with the process of solving \eqref{asymptotic outage probability of Dn with pSIC}, the expression for asymptotic outage probability of $D_m$ for PRIS-ARIS-NOMA under cascaded Rician fading and Nakagami-$m$ fading channels can be expressed as \eqref{asymptotic outage probability of Dm} at the top of the next page.
\begin{figure*}[!t]
\normalsize
\begin{small}
\begin{align}\label{asymptotic outage probability of Dm}
P_{out,{D_m}}^\infty  = \frac{{\pi {\mu _b}{{\left( {\overline {{\lambda _m}} {\gamma _{t{h_m}}}} \right)}^M}{{\left( {\beta \sigma _a^2M{\Omega _{na}} + {\sigma ^2}} \right)}^M}}}{{2K\left( {2M} \right)!\left( {2N - 1} \right)!{{\left[ {\beta {P_s}\left( {{a_m} - {\gamma _{t{h_m}}}{a_n}} \right)} \right]}^M}}}{\left[ {\frac{{4\sqrt \pi  \Gamma \left( {2{m_{na}}} \right)\Delta \left( 0 \right)}}{{{e^\kappa }\Gamma \left( {{m_{na}}} \right)}}} \right]^{M + N}}\sum\limits_{k = 1}^K {{{\left[ {\frac{{\left( {{x_k}{\text{ + }}1} \right){\mu _b}}}{2}} \right]}^{2N - 2M - 1}}\sqrt {1 - {x_k^2}} }.
\end{align}
\end{small}
\hrulefill \vspace*{0pt}
\end{figure*}
\end{corollary}
\begin{proof} The proof process of this corollary can be referred to the proof process of \textbf{Corollary \ref{corollary3}} in \textbf{Appendix B}.
\end{proof}
\begin{remark} \label{remark3}
By substituting \eqref{asymptotic outage probability of Dm} into \eqref{Definition of the diversity order}, the diversity order of $D_m$ can be derived as $M$, which is related to the number of the ARIS's reflecting elements.
\end{remark}
\subsection{Delay-limited Transmission}
In delay-limited transmission scheme, the BS sends signals at a constant power but the system throughput is subject to the outage probability of each user \cite{Wireless_informations}.
Therefore, the system throughput for PRIS-ARIS-NOMA under cascaded Rician fading and Nakagami-$m$ fading channels can be expressed as
\begin{align}\label{delay limited}
R_\varpi ^{{\text{limited}}} = \left( {1 - P_{{D_n}}^{out,\varpi }} \right){R_n} + \left( {1 - P_{{D_m}}^{out}} \right){R_m},
\end{align}
where $\varpi$ means ipSIC/pSIC scheme, $P_{{D_n}}^{out,ipSIC}$, $P_{{D_n}}^{out,pSIC}$ and $P_{{D_m}}^{out}$ can respectively be obtained from \eqref{OP for Dn with ipSIC}, \eqref{OP for Dn with pSIC} and \eqref{OP for Dm}.

\section{Ergodic Data Rate Analysis}\label{ergodic rate analysis}
The ergodic data rate is another important performance evaluation metric for wireless communication systems.
It is the maximum rate at which the code of the transmitting signals can traverse all fading states, i.e., the maximum rate at which the system can transmit signals correctly in a fading channel.
And the definition of the ergodic data rate can be expressed as
\begin{align}\label{ergodic rate}
{R^{ergodic}} = \mathbb{E}\left[ {{{\log }_2}\left( {1 + \gamma } \right)} \right],
\end{align}
where $\gamma$ means the SINR for users.
In this section the ergodic data rate expressions for $D_n$ with ipSIC/pSIC and $D_m$ are derived in detail. And we further derived the ergodic data rate slope expressions of $D_n$ and $D_m$.
\subsection{The Ergodic Data Rates of $D_n$}
When solving for the ergodic data rate of $D_n$, we assume that $D_n$ can successfully detect and decode the signals of $D_m$, then the ergodic data rate expression of $D_n$ with ipSIC can be given in the following theorem.
\begin{theorem} \label{theorem:4}
Under cascaded Rician fading and Nakagami-$m$ fading channels, the closed-form expression for ergodic data rate of $D_n$ with ipSIC for PRIS-ARIS-NOMA can be expressed as
\begin{small}
\begin{align}\label{ergodic for Dn with ipSIC}
  R_{Dn,ipSIC}^{ergodic} =& \sum\limits_{p = 1}^P {\sum\limits_{i = 1}^I {\sum\limits_{u = 1}^U {\frac{{{H_p}{H_i}{H_u}x_u^{{a_X}}x_p^{{a_Y}}}}{{\ln 2\Gamma ({a_X} + 1)\Gamma ({a_Y} + 1)}}} } }  \hfill \nonumber \\
   &\times \ln \left( {1 + \frac{{\beta {P_s^{a}}{a_n}b_X^2b_Y^2x_u^2x_p^2}}{{{{\bar \lambda }_n}\left( {\varpi {P_s^{a}}{\Omega _{RI}}{x_i} + \beta \sigma _a^2M{\Omega _{na}} + {\sigma ^2}} \right)}}} \right),
\end{align}
\end{small}where $x_p$ and $H_p$ are the parameters of Gauss-Laguerre quadrature formula and both of the two parameters share the same meaning with ${x_u},{x_i}$ and ${H_u},{H_i}$ ,respectively.
\begin{proof} See Appendix C.
\end{proof}
\end{theorem}
\begin{corollary} \label{corollary5}
When it comes to the special case of $\varpi  = 0$, the closed-formed expression for ergodic data rate of $D_n$ with pSIC for PRIS-ARIS-NOMA can be expressed as
\begin{align}\label{ergodic for Dn with pSIC}
  R_{Dn,ipSIC}^{ergodic} =& \sum\limits_{p = 1}^P {\sum\limits_{u = 1}^U {\frac{{{H_p}{H_u}x_u^{{a_X}}x_p^{{a_Y}}}}{{\ln 2\Gamma ({a_X} + 1)\Gamma ({a_Y} + 1)}}} }  \hfill \nonumber \\
   &\times \ln \left( {1 + \frac{{\beta P_s^a{a_n}b_X^2b_Y^2x_u^2x_p^2}}{{{{\bar \lambda }_n}\left( {\beta \sigma _a^2M{\Omega _{na}} + {\sigma ^2}} \right)}}} \right).
\end{align}
\end{corollary}
\subsection{The Ergodic Data Rates of $D_m$}
Similar to the method used for solving the ergodic data rate of $D_n$, the ergodic data rate expression of $D_m$ can be given in the following theorem.
\begin{theorem} \label{theorem:5}
Under cascaded Rician fading and Nakagami-$m$ fading channels, the closed-form expression for ergodic data rate of $D_m$ for PRIS-ARIS-NOMA can be expressed as
\begin{align}\label{ergodic for Dm}
  &R_{{D_m}}^{ergodic} = \sum\limits_{k = 1}^K {{\delta _k}\left[ {1 - \sum\limits_{u = 1}^U {\frac{{{H_u}x_u^{{a_X}}}}{{\Gamma ({a_X} + 1)\Gamma ({a_Y} + 1)}}} } \right.}  \hfill \nonumber \\
   &\times \left. {\gamma \left( {{a_Y} + 1,\frac{{\sqrt {\overline {{\lambda _m}} \left( {\beta \sigma _a^2M{\Omega _{na}} + {\sigma ^2}} \right)\left( {{x_k} + 1} \right){a_m}} }}{{\sqrt {{a_m}{a_n}\beta {P_s^{a}}\left( {1 - {x_k}} \right)} {b_X}{b_Y}{x_u}}}} \right)} \right],
\end{align}
where ${\delta _k} = \frac{{\pi {a_m}\sqrt {1 - x_k^2} }}{{K\ln 2\left[ {2{a_n} + {a_m}\left( {{x_k}{\text{ + }}1} \right)} \right]}}$.
\begin{proof} See Appendix D.
\end{proof}
\end{theorem}
\subsection{The Ergodic Data Rates of OMA benchmark}
\begin{corollary} \label{theorem:5}
Under Rician fading and Nakagami-$m$ fading cascaded channels, the closed-form expression for ergodic rate of the $\varphi$-th user for double RISs assisted OMA networks can be expressed as
\begin{align}\label{ergodic for OMA}
  R_{{D_\varphi },OMA}^{ergodic} =& \sum\limits_{p = 1}^P {\sum\limits_{u = 1}^U {\frac{{{H_p}{H_u}x_u^{{a_X}}x_p^{{a_Y}}}}{{2\ln 2\Gamma ({a_X} + 1)\Gamma ({a_Y} + 1)}}} }  \hfill \nonumber \\
   &\times \ln \left( {1 + \frac{{\beta P_s^a{a_\varphi }b_X^2b_Y^2x_u^2x_p^2}}{{{{\bar \lambda }_\varphi }\left( {\beta \sigma _a^2M{\Omega _{na}} + {\sigma ^2}} \right)}}} \right).
\end{align}
\end{corollary}
\subsection{Slope Analysis}
The ergodic data rate slope is another important metric for the performance of the networks to assess the evolution of ergodic data rate with the transmitting SNR, which can be captured from the ergodic data rate at high SNR region and defined as
\begin{align}\label{ergodic slope}
S = \mathop {\lim }\limits_{{P_s^{a}} \to \infty } \frac{{{R^{ergodic,\infty }}\left( {{P_s^{a}}} \right)}}{{\log {P_s^{a}}}},
\end{align}
where ${{R^{ergodic,\infty }}\left( {{P_s^{a}}} \right)}$ represents the asymptotic ergodic data rate at high SNR region.
\subsubsection{Slope analysis of $D_n$ with ipSIC}
On the basis of \eqref{ergodic for Dn with ipSIC}, the expression of asymptotic ergodic data rate of $D_n$ with ipSIC for PRIS-ARIS-NOMA at high SNR region, i.e. ${{P_s} \to \infty }$, can be derived as
\begin{small}
\begin{align}\label{asymptotic ergodic rate for Dn with ipSIC}
R_{{D_n},ipSIC}^{ergodic,\infty } = \sum\limits_{p = 1}^P {\sum\limits_{i = 1}^I {\sum\limits_{u = 1}^U {\frac{{{H_p}{H_i}{H_u}x_u^{{a_X}}x_p^{{a_Y}}\ln \left( {1 + \frac{{\beta {a_n}b_X^2b_Y^2x_u^2x_p^2}}{{{{\bar \lambda }_n}{\Omega _{RI}}{x_i}}}} \right)}}{{\ln 2\Gamma ({a_X} + 1)\Gamma ({a_Y} + 1)}}} } }.
\end{align}
\end{small}
\begin{remark} \label{remark4}
By substituting \eqref{asymptotic ergodic rate for Dn with ipSIC} into \eqref{ergodic slope}, a zero ergodic data rate slope for $D_n$ with ipSIC can be derived, which means that the residual interference during the ipSIC process limits the user's performance and causes the ergodic data rate to converge to a constant value when SNR increases.
\end{remark}
\subsubsection{Slope analysis of $D_n$ with pSIC}
With the help of Jensen's inequality written as $E\left[ {{{\log }_2}\left( {1 + {\gamma _{{D_n}}}} \right)} \right] \leqslant {\log _2}\left[ {1 + E\left( {{\gamma _{{D_n}}}} \right)} \right]$, we can obtain the upper bound of ergodic data rate of $D_n$ with pSIC for PRIS-ARIS-NOMA when ${{P_s} \to \infty }$ as
\begin{align}\label{upper bound for Dn with pSIC}
  R_{{D_n},pSIC}^{ergodic,up} &= {\log _2}\left[ {1 + \mathbb{E}\left( {\frac{{\beta {P_s^{a}}{{\left| {{\mathbf{h}}_2^T{{\mathbf{\Phi }}_p}{{\mathbf{h}}_1}} \right|}^2}{{\left| {{\mathbf{g}}_n^T{{\mathbf{\Phi }}_a}{{\mathbf{g}}_1}} \right|}^2}{a_n}}}{{\beta \sigma _a^2{{\left\| {{\mathbf{g}}_n^T{{\mathbf{\Phi }}_a}} \right\|}^2} + {\sigma ^2}}}} \right)} \right] \hfill \nonumber \\
   &= {\log _2}\left[ {1 + \frac{{{\vartheta _n}\beta {P_s^{a}}{a_n}{{\left( {\overline {{\lambda _n}} } \right)}^{ - 1}}}}{{\beta \sigma _a^2M{\Omega _{na}} + {\sigma ^2}}}} \right],
\end{align}
where ${\vartheta _n} = \left[ {{{\left( {N{e_\varphi }} \right)}^2} + N{d_\varphi }} \right]\left[ {{{\left( {M{e_\varphi }} \right)}^2} + M{d_\varphi }} \right]$, and ${{e_\varphi }}$ and ${{d_\varphi }}$ can be obtained from \eqref{Mean} and \eqref{Variance}.
\begin{remark} \label{remark5}
By substituting \eqref{upper bound for Dn with pSIC} into \eqref{ergodic slope}, we can obtain that the ergodic data rate slope for $D_n$ with pSIC is equal to $1$.
\end{remark}
\subsubsection{Slope analysis of $D_m$}
On the basis of \eqref{ergodic for Dm}, the expression of asymptotic ergodic data rate of $D_m$ for PRIS-ARIS-NOMA in the high SNR region, i.e. ${{P_s^{a}} \to \infty }$, can be derived as
\begin{align}\label{asymptotic ergodic rate for Dm}
  R_{{D_m}}^{ergodic,\infty } =& \sum\limits_{k = 1}^K {\frac{{\pi {a_m}\sqrt {1 - x_k^2} }}{{K\ln 2\left[ {2{a_n} + {a_m}\left( {{x_k}{\text{ + }}1} \right)} \right]}}}  \hfill \nonumber \\
   &\times \left[ {1 - \sum\limits_{u = 1}^U {\frac{{{H_u}x_u^{{a_X}}}}{{\Gamma ({a_X} + 1)\Gamma ({a_Y} + 1)}}} } \right].
\end{align}
\begin{remark} \label{remark6}
By substituting \eqref{asymptotic ergodic rate for Dm} into \eqref{ergodic slope}, we can obtain that the ergodic data rate slope for $D_m$ is equal to zero, which is the same conclusion as that obtained by \textbf{Remark \ref{remark4}}.
\end{remark}
\subsection{Delay-tolerant Transmission}
In delay-tolerant mode, the source can transmit information at any ergodic data rate with an upper bound on it and the signals can fulfill each state of the channel's traversal during this mode \cite{Wireless_informations}.
Therefore, the system throughput for PRIS-ARIS-NOMA under cascaded Rician fading and Nakagami-$m$ fading channels can be expressed as
\begin{align}\label{delay tolerant}
R_\varpi ^{{\text{tolerant}}} = R_{{D_n},\varpi }^{ergodic} + R_{{D_m}}^{ergodic}.
\end{align}
where $R_{Dn,ipSIC}^{ergodic}$, $R_{Dn,pSIC}^{ergodic}$ and $R_{{D_m}}^{ergodic}$ can respectively be obtained from \eqref{ergodic for Dn with ipSIC}, \eqref{ergodic for Dn with pSIC} and \eqref{ergodic for Dm}.
\begin{table}[!h]
\caption{The fixed numerical values of the parameters.}
\begin{center}
{\tabcolsep12pt\begin{tabular}{|l|l|}\hline   
\cline{1-2}
Monte Carlo simulation repeated & $10^6$ iterations \\
\cline{1-2}
Rician factor & $\kappa {\text{ =  - 5dB}}$ \\
\cline{1-2}
Shaping parameter & ${m_{na}} = 1$ \\
\cline{1-2}
Amplification factor & $\beta {\text{ = 2}}{\text{.5}}$ \\
\cline{1-2}
Number of reflecting elements & $M = N = 3$ \\
\cline{1-2}
Communication link distance & $\begin{gathered}
                               {d_{h1}} = {d_{h2}} = 10{\text{ m}} \hfill \\
                               {d_{g1}} = {d_{gn}} = 20{\text{ m}} \hfill \\
                               {d_{gm}} = 80{\text{ m}} \hfill \\
                               \end{gathered}  $  \\
\cline{1-2}
Two users'power allocations  & $\begin{gathered}
                                    {a_n} = 0.2 \hfill \\
                                    {a_m} = 0.8 \hfill \\
                                    \end{gathered} $   \\
\cline{1-2}
Two users'target rates       & $\begin{gathered}
                                {R_n} = 2{\text{ BPCU}} \hfill \\
                                {R_m} = 2{\text{ BPCU}} \hfill \\
                                \end{gathered}  $ \\
\cline{1-2}
Noise power & $\begin{gathered}
               \sigma _a^2 =  - 80{\text{ dbm}} \hfill \\
              {\sigma ^2} =  - 70{\text{ dbm}} \hfill \\
               \end{gathered} $  \\
\cline{1-2}
Path loss factors & $\begin{gathered}
                     \alpha {\text{ = 2}} \hfill \\
                     \eta {\text{ =  - 10 dbm}} \hfill \\
                     \end{gathered} $ \\
\hline
\end{tabular}}{}
\label{tab1}
\end{center}
\end{table}
\section{Simulation And Numerical Results}\label{Simulation and numerical results}
In this section, simulations are provided to verify the accurateness of equations derived form the above sections.
The fixed numerical values of the parameters are indicated in TABLE I, and BPCU is the abbreviation of bit per channel use.
To show the enhancement of PRIS-ARIS-NOMA, PRIS-ARIS-OMA and double PRISs-NOMA are presented as benchmark.
We have borrowed the simulation approach from \cite{You_mul} and \cite{Yuanwei_NOMA_RIS}, and combined it with the model in this paper, and the validation results of the numerical part are similar to those of these two papers, thus verifying the feasibility of communication in this scenario.
For OMA scheme, we assume the transmission of users information is orthogonal in the frequency domain, and each of the two users occupies half of the frequency domain resource.
To ensure fairness, the total power consumption of each system is meant to be the same.
Specifically, the total power consumption of PRIS-ARIS-NOMA and double PRISs-NOMA are respectively presented as $Q_{total}^{active} = P_s^{a} + {P_{aris}} + \left( {M + N} \right){P_{sw}} + M{P_{dc}}$ and $Q_{total}^{passive} = P_s^p + \left( {M + N} \right){P_{sw}}$, where $P_s^p$ means the transmitting power of the base station for double PRISs-NOMA, ${P_{aris}}$ means the output signal power of ARIS while ${P_{sw}}$ and ${P_{dc}}$ respectively represent the power consumption of the phase shift switch and control circuit in each reflecting element and the DC bias power of each reflecting element on ARIS \cite{Pancunhua}.
Then it can be defined as ${P_{to}} = Q_{total}^{active} = Q_{total}^{passive}$.
The complexity-accuracy trade-off parameters are set to be $P = I = U = 500$, $K = 100$.
For simulation parameter settings, it can be referred from \cite{Yue_full_half_NOMA}, \cite{yue2021RISNOMA}, and \cite{STARS_Yue}.
\subsection{Outage Probability}
Fig. 2 depicts the outage probability of PRIS-ARIS-NOMA versus the transmitting power of the BS, and also compares the outage probability of $D_n$ with different residual interference power under the ipSIC scheme and the outage probability of the two users under OMA networks.
As shown in the figure, the curves of outage probability for $D_n$ with ipSIC/pSIC and $D_m$ can be plotted in terms of \eqref{OP for Dn with ipSIC}, \eqref{OP for Dn with pSIC} and \eqref{OP for Dm}, while the curves of outage probability for user $D_n$ and user $D_m$ under OMA networks can be plotted in terms of \eqref{OP for OMAbenchmark}.
And the curves for asymptotic outage probability can be plotted in terms of \eqref{asymptotic outage probability of Dn with ipSIC}, \eqref{asymptotic outage probability of Dn with pSIC} and \eqref{asymptotic outage probability of Dm}.
It can be observed from the figure that the theoretical value curve is highly coincident with the simulation, thus verifying the correctness and applicability of the theoretical derivation.
For the user's outage performance, it can be read from the picture that the NOMA networks outperform the OMA networks and reflect the fact that NOMA can provide better user fairness than OMA when there are channel differences between users \cite{SIC}.
Consistent with \textbf{Remark \ref{remark1}}, the outage probability of $D_n$ with ipSIC eventually gets to an error floor for the impact of the residual interference, and $D_n$ with ipSIC scheme will get a better outage performance as the residual interference power decreases.
Hence, it is crucial to take the residual interference into consideration when it comes to a practical communication scenario.
\begin{figure}[t!]
    \begin{center}
        \includegraphics[width=3.2in,  height=2.5in]{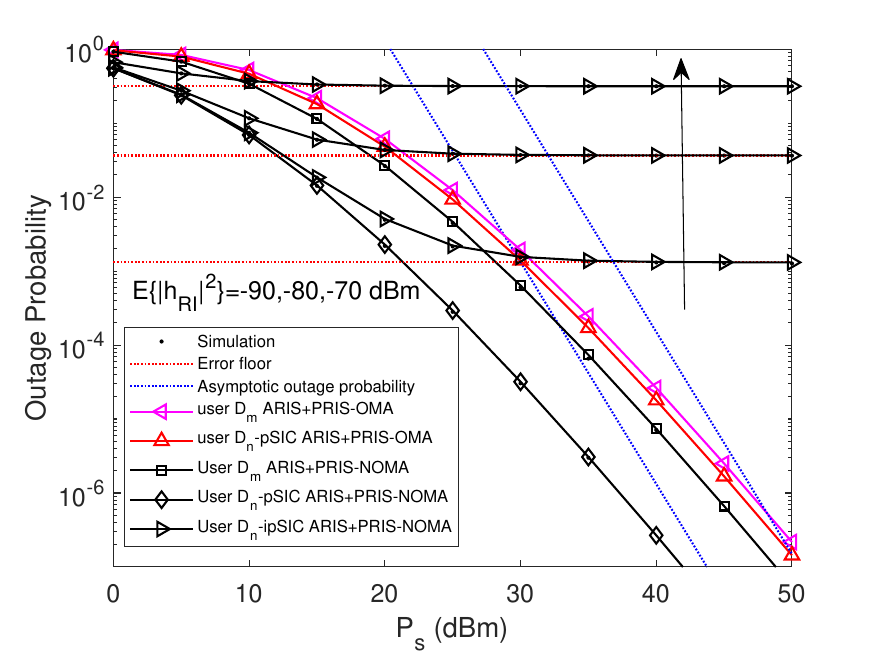}
        \caption*{Fig. 2: Outage probability versus the transmitting power of the BS for PRIS-ARIS-NOMA.}
        \label{P2}
    \end{center}
\end{figure}

Fig. 3 depicts the outage probability for the comparison of PRIS-ARIS-NOMA and double PRISs-NOMA versus the total power consumption of the networks with settings of $\Omega_{RI}=-80$ dBm.
In this case, the outage probability of users without RISs assistance scenario is also plotted in the figure, and it can be seen that the user outage performance of the communication link without RISs assistance is poor in this communication environment.
As can be observed from the figure that the outage performance of PRIS-ARIS networks is better than the double PRISs networks, the cause of this phenomenon is that ARIS is equipped with active bias circuitry compared to PRIS and has a built-in signal amplifier to increase the signal power, which helps improve the user's reception SNR at the receiving end.
It also reflects that the use of ARIS can better combat fading in the channel with the comparison to PRIS.
\begin{figure}[t!]
    \begin{center}
        \includegraphics[width=3.2in,  height=2.5in]{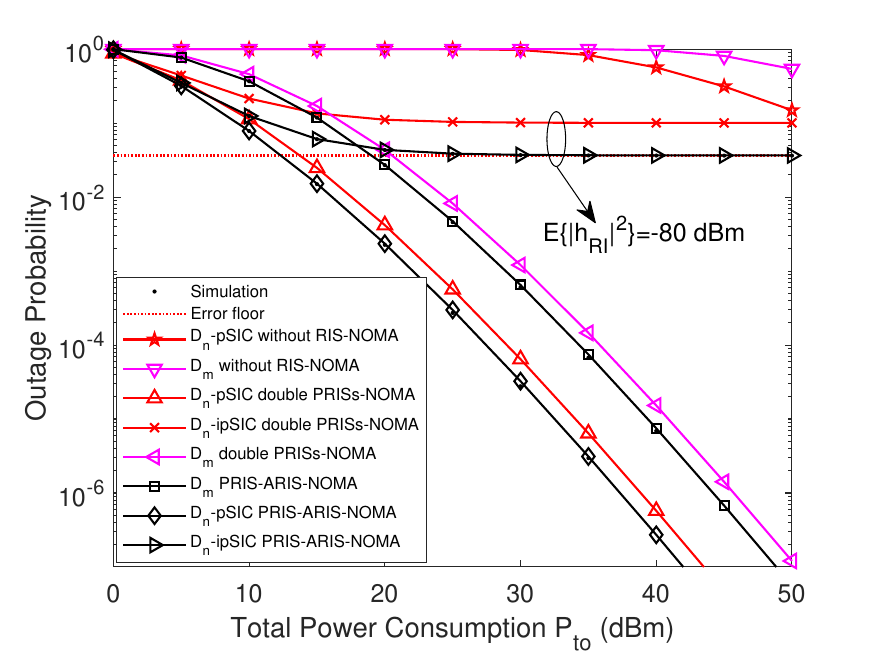}
        \caption*{Fig. 3: Outage probability versus the total power consumption.}
        \label{P3}
    \end{center}
\end{figure}
\begin{figure}[t!]
    \begin{center}
        \includegraphics[width=3.2in,  height=2.5in]{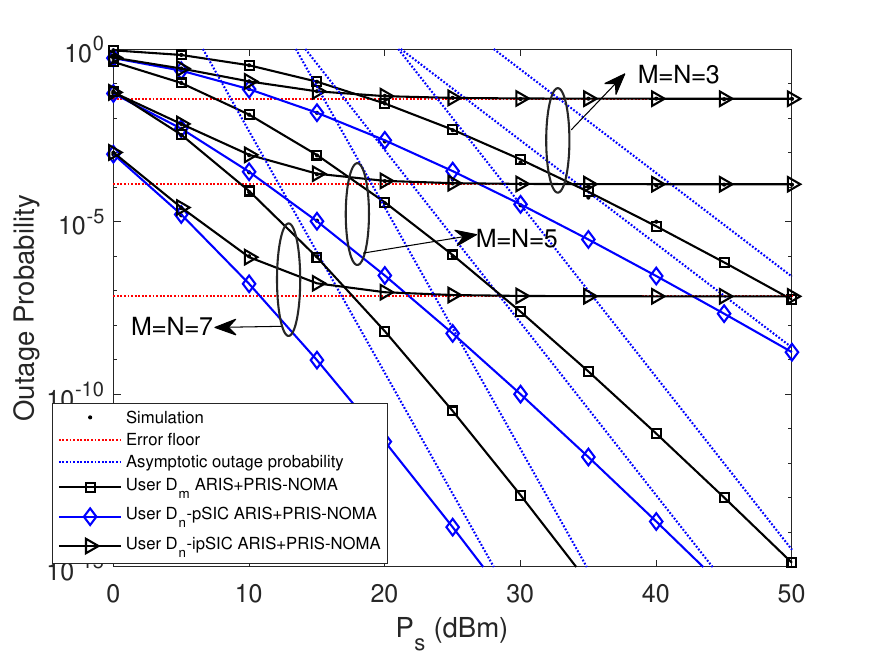}
        \caption*{Fig. 4: Outage probability versus the transmitting power of the BS for PRIS-ARIS-NOMA.}
        \label{P4}
    \end{center}
\end{figure}

Fig. 4 depicts the outage probability of PRIS-ARIS-NOMA versus the transmitting power of the BS with settings of $\Omega_{RI}=-80$ dBm, and compares users outage probability with different reflecting elements.
From the trend of outage probability in the figure, we can find that as the number of RISs reflecting elements increases, the users will tend to get a better outage performance.
This is because that the number of reflecting elements determines the number of independent divisions of the transmission, and as the number of independent divisions increases, the ability of the signal to fight against channel fading increases by a certain degree, so the users outage performance becomes better.
Meanwhile, in line with \textbf{Remark \ref{remark2}} and \textbf{Remark \ref{remark3}}, the diversity orders of $D_n$ and $D_m$ are affected by the reflecting elements on RISs. Therefore, as the number of reflecting elements $M$ increases, the users outage probability curves gain a larger slope and decrease faster.
\begin{figure}[t!]
    \begin{center}
        \includegraphics[width=3.2in,  height=2.5in]{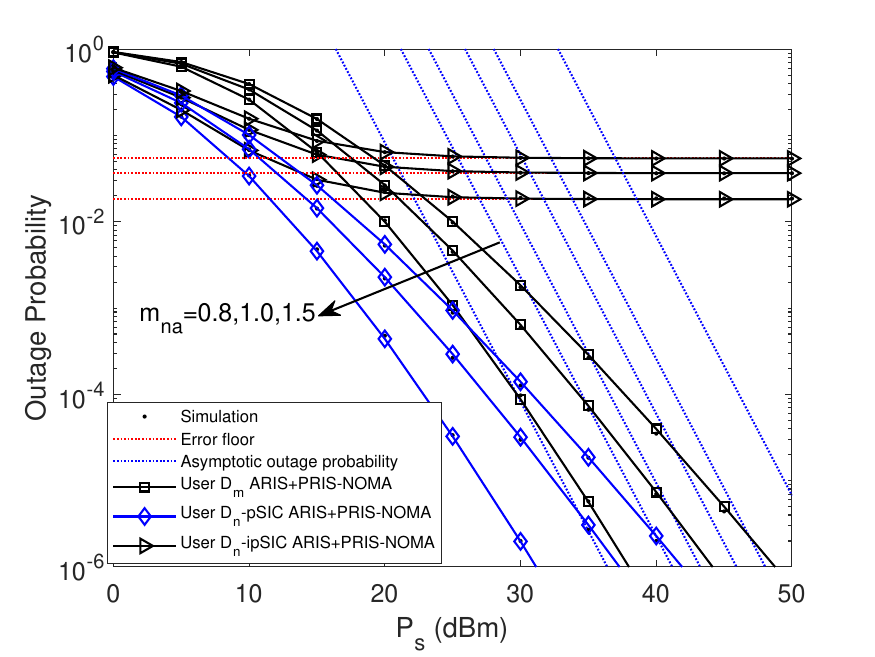}
        \caption*{Fig. 5: Outage probability versus the transmitting power of the BS for PRIS-ARIS-NOMA.}
        \label{P5}
    \end{center}
\end{figure}

Fig. 5 depicts the outage probability of PRIS-ARIS-NOMA versus the transmitting power of the BS with settings of $\Omega_{RI}=-80$ dBm, and compares users outage probability with different shaping parameter $m_{na}$ of Nakagami-$m$ random variables.
From the figure, it can be intuitively observed that the outage performance of users gradually becomes better as $m_{na}$ increases. This is because the physical meaning of $m_{na}$ in the communication scenario represents the depth of channel fading, with larger $m_{na}$ indicating shallower channel fading and higher communication quality of the communication channels.
Theoretically, when $m_{na}$ tends to 0, it means communication is almost impossible to achieve, and when $m_{na}$ tends to infinity, it means no fading exists in the communication channels.
It can also be observed from the figure that the variation of $m_{na}$ does not affect the slope of the users outage probability, which is due to the fact that the user's diversity order depends only on the reflecting elements number of ARIS, which coincides with the \textbf{Remark \ref{remark2}} and \textbf{Remark \ref{remark3}}.

\begin{figure}[t!]
    \begin{center}
        \includegraphics[width=3.2in,  height=2.5in]{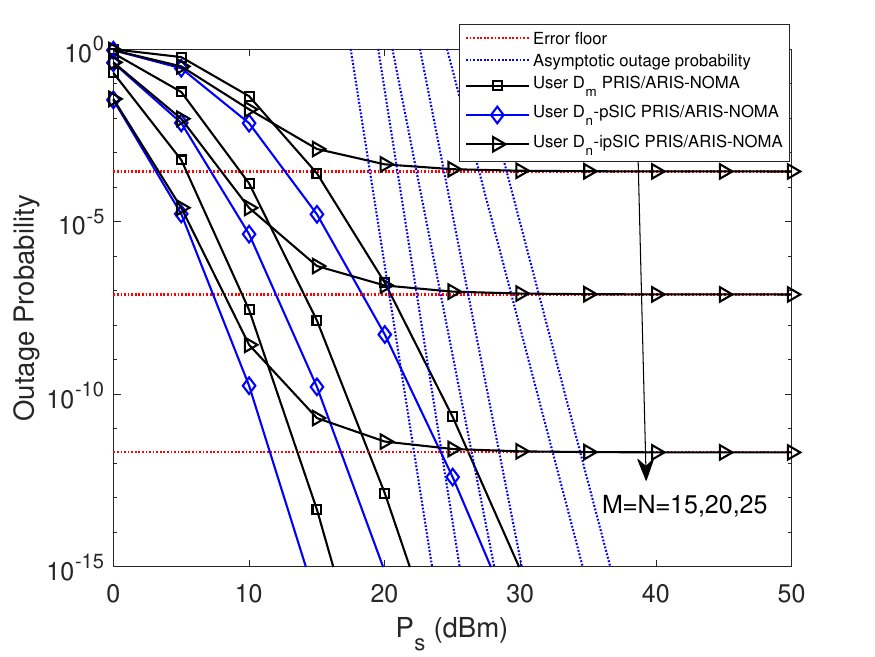}
        \caption*{Fig. 6: Outage probability versus the transmitting power of the BS for PRIS-ARIS-NOMA over long-distance transmission scenario.}
        \label{P6}
    \end{center}
\end{figure}

In order to demonstrate the advantages of this system model for long-distance transmission, we also analyzed a scenario of PRIS-ARIS-NOMA networks for long-distance transmission as shown in Fig. 6.
In this scenario, the distance between the communication links is set as $d_{h1}=d_{h2}=d_{g1}= 1000{\text{ m}}$, $d_{gn} = 800{\text{ m}}$, $d_{gm} = 1800{\text{ m}}$ respectively, and the number of reflecting elements used by the two RISs grows moderately as well.
From the figure, it can be seen that the PRIS-ARIS-NOMA networks has a significant enhancement of the communication performance for long-distance transmission, and the users outage probabilities can reach a threshold of $10^{-3}$ at a transmitting power around 5 dBm.
Also as the number of reflecting elements increases, the outage probability decreases more rapidly, which confirms the insights obtained from the \textbf{Remark \ref{remark2}} and \textbf{Remark \ref{remark3}}.
\subsection{Ergodic Data Rate}
Fig. 7 depicts the ergodic data rate of PRIS-ARIS-NOMA versus the transmitting power of the BS with settings of $\Omega_{RI}=-80$ dBm, and also compares the ergodic data rate of $D_n$ and $D_m$ under OMA networks.
As shown in the figure, the curves of ergodic data rate for $D_n$ with ipSIC/pSIC and $D_m$ can be plotted in terms of \eqref{ergodic for Dn with ipSIC}, \eqref{ergodic for Dn with pSIC} and \eqref{ergodic for Dm}, while the curves of ergodic rate for user $D_n$ and user $D_m$ under OMA networks can be plotted in terms of \eqref{ergodic for OMA}.
And the curves for asymptotic ergodic data rate can be plotted in terms of \eqref{asymptotic ergodic rate for Dn with ipSIC}, \eqref{upper bound for Dn with pSIC} and \eqref{asymptotic ergodic rate for Dm}.
The figure shows that the ergodic data rate of $D_n$ with ipSIC and $D_m$ eventually converge to an upper throughput limit and converge to a zero ergodic data rate slope at high SNR region, which are in lines with the \textbf{Remark \ref{remark4}} and \textbf{Remark \ref{remark6}}.
This is because for $D_n$ with ipSIC, it is similar to the analysis of the outage probability, self interference affects the users receiving performance more as the transmitting power increases. And for $D_m$, upper throughput limit is only dependent on the power allocation factors of the two users according to \eqref{asymptotic ergodic rate for Dm}.
Moreover, the ergodic performance of $D_n$ with ipSIC under NOMA networks is better than that of the OMA networks.
For OMA networks, the division of resources for per user in the frequency domain is reduced by one-half compared to NOMA systems, so the ergodic data rate slope of the OMA networks is reduced by one-half compared to NOMA networks, which is the reason why the ergodic data rate under OMA networks grows slower than that of the NOMA networks.
\begin{figure}[t!]
    \begin{center}
        \includegraphics[width=3.2in,  height=2.5in]{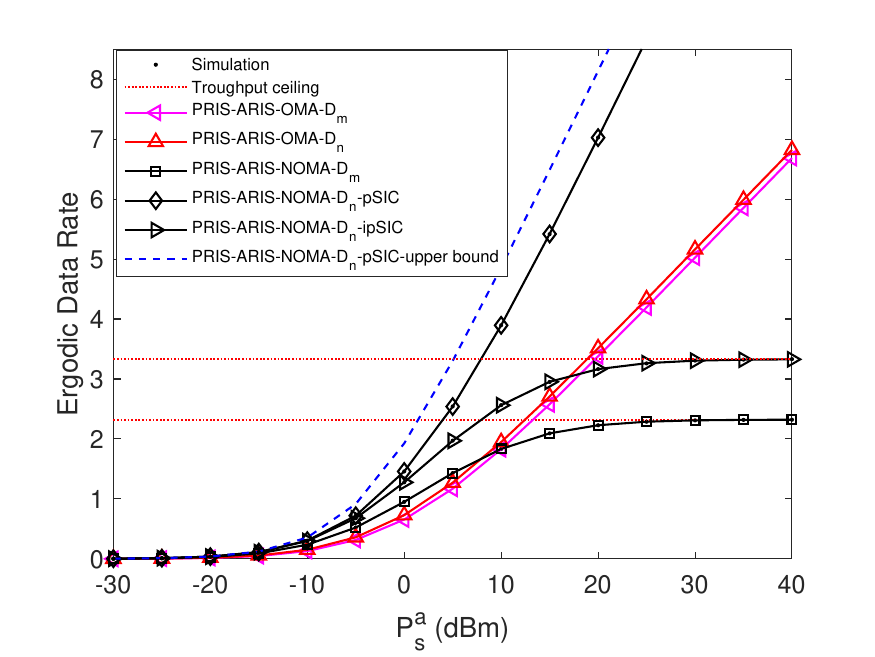}
        \caption*{Fig. 7: Ergodic data rate versus the transmitting power of the BS for PRIS-ARIS-NOMA.}
        \label{P7}
    \end{center}
\end{figure}
\begin{figure}[t!]
    \begin{center}
        \includegraphics[width=3.2in,  height=2.5in]{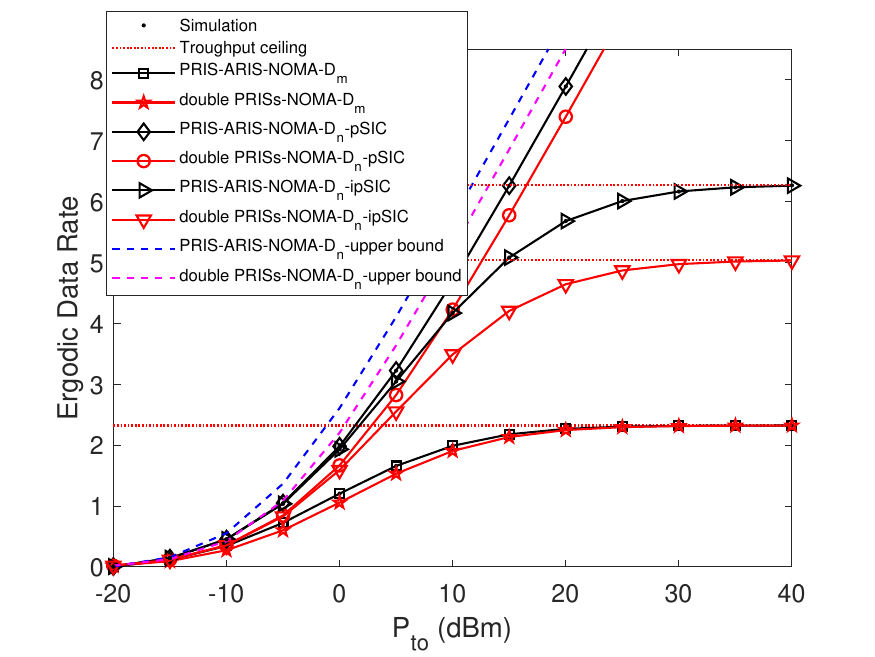}
        \caption*{Fig. 8: Ergodic data rate versus the total power consumption.}
        \label{P8}
    \end{center}
\end{figure}

Fig. 8 depicts the ergodic data rate for the comparison of PRIS-ARIS-NOMA and double PRISs-NOMA versus the total power consumption of the networks with settings of $\Omega_{RI}=-80$ dBm.
One can be seen is that the ergodic data rates of $D_n$ with ipSIC/pSIC and $D_m$ in PRIS-ARIS-NOMA are superior to that in double PRISs-NOMA, this is because that ARIS amplifies and forwards the signals through the built-in bias circuit, which can increase the SINR at the receiving end, so the ergodic data rate will be higher.
Fig. 9 depicts the ergodic data rate for PRIS-ARIS-NOMA versus the transmitting power of the BS with settings of $\Omega_{RI}=-80$ dBm, and compares the ergodic data rate of the two users under different number of reflecting elements of the ARIS. From the figure we can observe that as the number of ARIS reflecting elements increases, the ergodic data rate gradually increases.
\begin{figure}[t!]
    \begin{center}
        \includegraphics[width=3.2in,  height=2.5in]{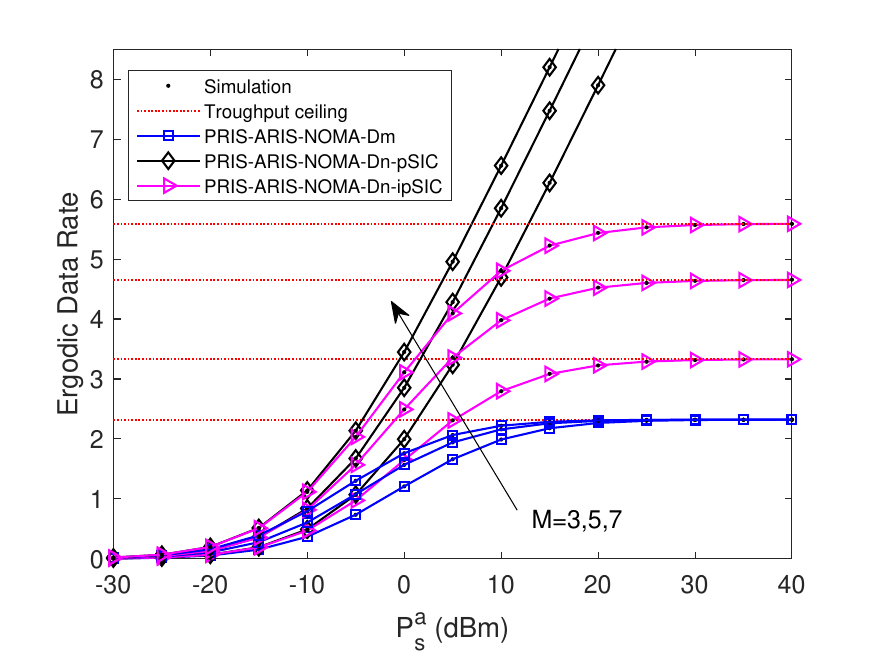}
        \caption*{Fig. 9: Ergodic data rate versus the transmitting power of the BS for PRIS-ARIS-NOMA.}
        \label{P9}
    \end{center}
\end{figure}
\subsection{System Throughput}
\begin{figure}[t!]
    \begin{center}
        \includegraphics[width=3.2in,  height=2.5in]{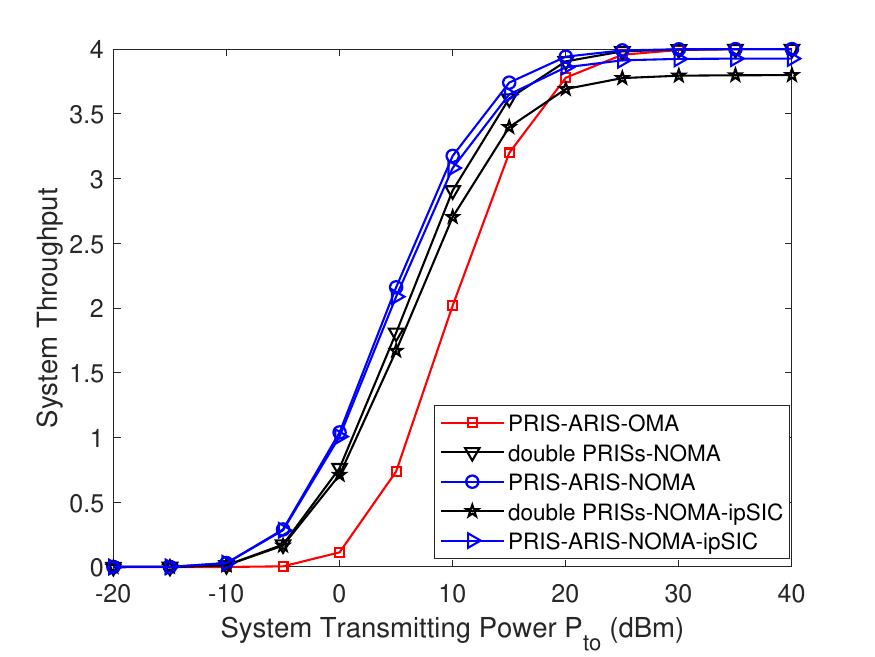}
        \caption*{Fig. 10: System throughput versus the total power consumption.}
        \label{P10}
    \end{center}
\end{figure}
Fig. 10 depicts the system throughput for the comparison of PRIS-ARIS-NOMA and double PRISs-NOMA versus the total power consumption of the networks in delay-limited transmission mode with settings of $\Omega_{RI}=-80$ dBm.
The curves for the system throughput in delay-limited transmission mode can be drawn in terms of \eqref{delay limited}.
It can be observed from the figure that the system throughput of the PRIS-ARIS-NOMA in delay-limited transmission mode is higher than that of the double PRISs networks and OMA networks.
This is due to the fact that the system throughput in delay-limited transmission mode is relevant to the system outage probability and a lower outage probability yields a higher throughput.
Another phenomenon is that the system throughput with ipSIC scheme cannot reach the maximum value.
This is because residual interference limits the extent to which the outage probability decreases with increasing transmitting power, thus producing an upper limit on the system throughput that does not reach the expected maximum value.
\begin{figure}[t!]
    \begin{center}
        \includegraphics[width=3.2in,  height=2.5in]{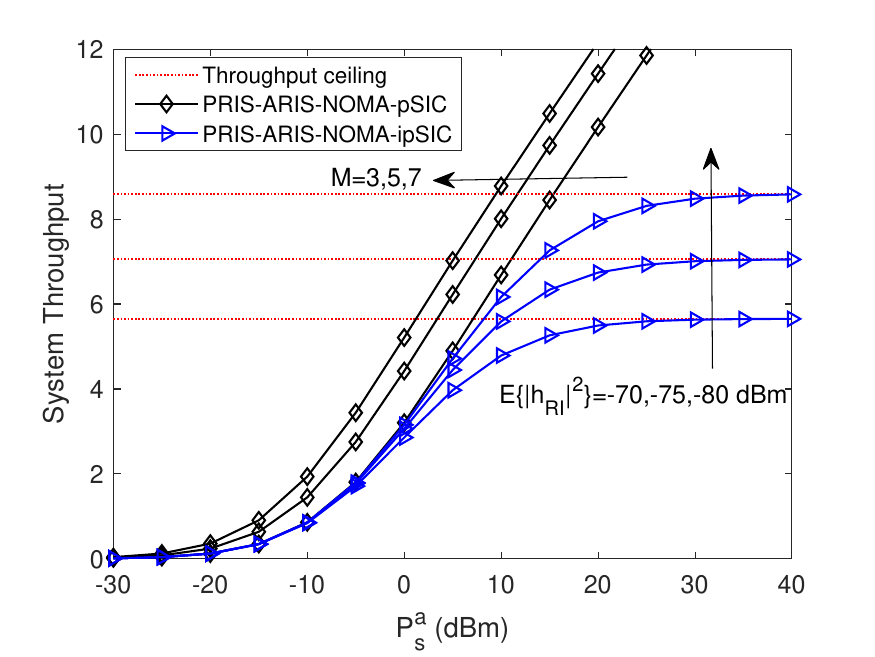}
        \caption*{Fig. 11: System throughput versus the transmitting power of the BS for PRIS-ARIS-NOMA.}
        \label{P11}
    \end{center}
\end{figure}
Fig. 11 depicts the the system throughput for PRIS-ARIS-NOMA versus the transmitting power of the BS in delay-tolerant transmission mode with settings of $\Omega_{RI}=-80$ dBm.
The curves for the system throughput in delay-tolerant transmission mode can be drawn in terms of \eqref{delay tolerant}.
From the figure, we can observe that the throughput of the system can be improved by increasing the number of RIS reflecting elements and reducing the power of self interference under the ipSIC scheme.
It also reflects the importance of optimizing the SIC process in NOMA networks.
\section{Conclusion}\label{conclusion}
In this paper, the PRIS-ARIS assisted NOMA networks under cascaded Rician fading and Nakagami-$m$ fading channels have been studied in detail.
Specifically, the closed-form expressions for outage probability and ergodic data rate of a pair of two non-orthogonal users have been derived in detail, as well as the asymptotic expressions for outage probability and ergodic data rate are deduced for more insights.
Moreover, diversity orders and ergodic data rate slopes are analyzed based on the asymptotic expression to evaluate the performance of the networks.
Furthermore, the system throughput in delay-limited mode and delay-tolerant mode are evaluated to show the system enhancement.
Eventually, through detailed study of this paper, we can find that PRIS-ARIS-NOMA outperforms PRIS-ARIS-OMA and double PRISs-NOMA.
Compared to PRIS, the use of ARIS at user end can better combat channel fading under the cascaded channel scenario.
And as the number of RISs reflecting elements increases within a certain range, the networks tend to perform better.

\section*{Appendix~A: Proof of Theorem \ref{theorem:1}} \label{Appendix:A}
\renewcommand{\theequation}{A.\arabic{equation}}
\setcounter{equation}{0}
To prove this theorem, we started by bringing \eqref{SINRND} and \eqref{SINRD} into \eqref{outage event Dn}.
The outage probability expression of $D_n$ with ipSIC can be further derived as \eqref{Appendix A 1} at the top of the next page. Then, simplifying \eqref{Appendix A 1} yields
\begin{figure*}[!t]
\normalsize
\begin{small}
\begin{align}\label{Appendix A 1}
  P_{{D_n}}^{ipSIC}=& \Pr \left( {\frac{{\beta {P_s^{a}}{{\left| {{\mathbf{h}}_2^T{{\mathbf{\Phi }}_p}{{\mathbf{h}}_1}} \right|}^2}{{\left| {{\mathbf{g}}_n^T{{\mathbf{\Phi }}_a}{{\mathbf{g}}_1}} \right|}^2}{a_m}}}{{\beta {P_s^{a}}{{\left| {{\mathbf{h}}_2^T{{\mathbf{\Phi }}_p}{{\mathbf{h}}_1}} \right|}^2}{{\left| {{\mathbf{g}}_n^T{{\mathbf{\Phi }}_a}{{\mathbf{g}}_1}} \right|}^2}{a_n} + \beta \sigma _a^2{{\left\| {{\mathbf{g}}_n^T{{\mathbf{\Phi }}_a}} \right\|}^2} + {\sigma ^2}}} < {\gamma _{t{h_m}}}} \right) \hfill \nonumber \\
   &+ \Pr \left( {\frac{{\beta {P_s^{a}}{{\left| {{\mathbf{h}}_2^T{{\mathbf{\Phi }}_p}{{\mathbf{h}}_1}} \right|}^2}{{\left| {{\mathbf{g}}_n^T{{\mathbf{\Phi }}_a}{{\mathbf{g}}_1}} \right|}^2}{a_m}}}{{\beta {P_s^{a}}{{\left| {{\mathbf{h}}_2^T{{\mathbf{\Phi }}_p}{{\mathbf{h}}_1}} \right|}^2}{{\left| {{\mathbf{g}}_n^T{{\mathbf{\Phi }}_a}{{\mathbf{g}}_1}} \right|}^2}{a_n} + \beta \sigma _a^2{{\left\| {{\mathbf{g}}_n^T{{\mathbf{\Phi }}_a}} \right\|}^2} + {\sigma ^2}}} > {\gamma _{t{h_m}}},\frac{{\beta {P_s^{a}}{{\left| {{\mathbf{h}}_2^T{{\mathbf{\Phi }}_p}{{\mathbf{h}}_1}} \right|}^2}{{\left| {{\mathbf{g}}_n^T{{\mathbf{\Phi }}_a}{{\mathbf{g}}_1}} \right|}^2}{a_n}}}{{\varpi {P_s^{a}}{{\left| {{h_{RI}}} \right|}^2} + \beta \sigma _a^2{{\left\| {{\mathbf{g}}_n^T{{\mathbf{\Phi }}_a}} \right\|}^2} + {\sigma ^2}}} < {\gamma _{t{h_n}}}} \right).
\end{align}
\end{small}
\hrulefill \vspace*{0pt}
\end{figure*}
\begin{small}
\begin{align}\label{OP for Dn with ipSIC step2}
  P_{{D_n}}^{out,ipSIC} &= \Pr \left[ {\underbrace {{{\left| {\sum\limits_{n = 0}^N {h_n^1} h_n^2} \right|}^2}}_{{\zeta _h}}\underbrace {{{\left| {\sum\limits_{m = 0}^M {g_m^1} g_m^\varphi } \right|}^2}}_{{\zeta _g}}} \right. \hfill \nonumber \\
  &\left. { < \frac{{\overline {{\lambda _n}} {\gamma _{t{h_n}}}\left( {\varpi {P_s^{a}}{{\left| {{h_{RI}}} \right|}^2} + \beta \sigma _a^2{{\left\| {{\mathbf{g}}_n^T{{\mathbf{\Phi }}_a}} \right\|}^2} + {\sigma ^2}} \right)}}{{\beta {P_s^{a}}{a_n}}}} \right],
\end{align}
\end{small}where, $\overline {{\lambda _n}}  = {\eta ^{ - 4}}d_{h1}^\alpha d_{h2}^\alpha d_{g1}^\alpha d_{gn}^\alpha $, and the PDF of $\zeta _h$ and $\zeta _g$ are denoted by \eqref{zetahpdf} and \eqref{zetagpdf} in Section II. The operation here needs to obtain the PDF of ${\zeta _h}{\zeta _g}$, and combine \eqref{zetahpdf} and \eqref{zetagpdf}, the PDF of $Z = {\zeta _h}{\zeta _g}$ can be derived by
\begin{align}\label{OP for Dn with ipSIC step3}
  {F_Z}\left( z \right) &= \int_0^\infty  {\left[ {\int_0^{\frac{z}{x}} {{f_{{\zeta _g}}}\left( y \right)} dy} \right]} {f_{{\zeta _h}}}\left( x \right)dx.
\end{align}

Performing a mathematical simplification and using Gauss-Laguerre quadrature formula \cite[Sec. 2.2.2]{Stochastic_Methods} to approximate the result of the integral expression yields
\begin{align}\label{OP for Dn with ipSIC step4}
{F_Z}\left( z \right) = \sum\limits_{u = 1}^U {\frac{{{H_u}x_u^{{a_X}}}}{{\Gamma ({a_X} + 1)\Gamma ({a_Y} + 1)}}\gamma \left( {{a_Y} + 1,\frac{{\sqrt z }}{{{b_X}{b_Y}{x_u}}}} \right)},
\end{align}
where $\gamma (\alpha ,x) = \int_0^x {{e^{ - t}}{t^{\alpha  - 1}}} dt$ means the lower incomplete gamma function\cite[Eq. (8.350.1)]{table}. ${{H_u}}$ and $x_u$ are the $u$-th weight and $u$-th zero point of Laguerre polynomial ${L_u}({x_u})$, respectively. Specifically, ${H_u} = \frac{{{{[(U + 1)!]}^2}}}{{{x_u}{{[{{L'}_{U + 1}}({x_u})]}^2}}},u = 0,1,...,U$ and ${L_\varphi}\left( x \right) = {e^x}\frac{{{d^\varphi}}}{{d{x^\varphi}}}\left( {{x^\varphi}{e^{ - x}}} \right)$.

In \eqref{OP for Dn with ipSIC step2}, the self interference part ${{{\left| {{h_{RI}}} \right|}^2}}$ follows an exponential distribution with parameter ${\Omega _{RI}}$.
Then, by substituting \eqref{OP for Dn with ipSIC step4} into \eqref{OP for Dn with ipSIC step2} and approximating thermal noise part ${{\left\| {{\mathbf{g}}_n^T{{\mathbf{\Phi }}_a}} \right\|}^2}$ to ${M{\Omega _{na}}}$, the expression of outage probability for $D_n$ with ipSIC can be derived by
\begin{align}\label{OP for Dn with ipSIC step5}
  P_{{D_n}}^{out,ipSIC} &= \int_0^{ + \infty } {\frac{1}{{{\Omega _{RI}}}}{e^{ - \frac{x}{{{\Omega _{RI}}}}}}}  \hfill \nonumber \\
   &\times {F_Z}\left[ {\frac{{\overline {{\lambda _n}} {\gamma _{t{h_n}}}\left( {\varpi {P_s^{a}}x + \beta \sigma _a^2M{\Omega _{na}} + {\sigma ^2}} \right)}}{{\beta {P_s^{a}}{a_n}}}} \right]dx.
\end{align}

Finally, by once again using the Gauss-Laguerre quadrature formula in \eqref{OP for Dn with ipSIC step5} and performing algebraic operations, \eqref{OP for Dn with ipSIC} can be obtained. The proof is completed.
\section*{Appendix~B: Proof of Corollary \ref{corollary3}} \label{Appendix:B}
\renewcommand{\theequation}{B.\arabic{equation}}
\setcounter{equation}{0}
Proof starts with substituting $\varpi  = 0$ into \eqref{OP for Dn with ipSIC step2} and the outage probability of $D_n$ with pSIC can be expressed as
\begin{small}
\begin{align}\label{OP for Dn with pSIC step1}
  P_{{D_n}}^{out,pSIC} =& \Pr \left[ {\underbrace {\left( {\sum\limits_{n = 0}^N {\left| {h_n^1h_n^2} \right|} } \right)}_{{{\hat \zeta }_h}}\underbrace {\left( {\sum\limits_{m = 0}^M {\left| {g_m^1g_m^\varphi } \right|} } \right)}_{{{\hat \zeta }_g}}} \right. \hfill \nonumber \\
  &\left. { < \sqrt {\frac{{\overline {{\lambda _n}} {\gamma _{t{h_n}}}\left( {\beta \sigma _a^2{{\left\| {{\mathbf{g}}_n^T{{\mathbf{\Phi }}_a}} \right\|}^2} + {\sigma ^2}} \right)}}{{\beta {P_s^{a}}{a_n}}}} } \right].
\end{align}
\end{small}

Then to derive the expression of the asymptotic outage probability, the next step is to obtain the approximate CDF expression of $\hat Z = {{\hat \zeta }_h}{{\hat \zeta }_g}$ at high SNR region.
Using the Laplace transform on \eqref{Single Cascade Channel PDF} with the help of the integral formula \cite[Sec. 6.621.3]{Stochastic_Methods} yields
\begin{small}
\begin{align}\label{OP for Dn with pSIC step2}
  \mathcal{L}&\left[ {{f_{{X_n}}}\left( y \right)} \right]\left( s \right) = \sum\limits_{m = 0}^\infty  {\frac{{4{\kappa ^m}{{\left[ {\left( {\kappa  + 1} \right){m_{na}}} \right]}^{\frac{1}{2}\left( {1 + m + {m_{na}}} \right)}}}}{{{e^\kappa }{{\left( {m!} \right)}^2}\Gamma \left( {{m_{na}}} \right)}}}  \hfill \nonumber \\
   &\times \frac{{\sqrt \pi  {{\left( {4\sqrt {\left( {\kappa  + 1} \right){m_{na}}} } \right)}^{m - {m_{na}} + 1}}}}{{{{\left( {s + 2\sqrt {\left( {\kappa  + 1} \right){m_{na}}} } \right)}^{2m + 2}}}}\frac{{\Gamma \left( {2m + 2} \right)\Gamma \left( {2{m_{na}}} \right)}}{{\Gamma \left( {m + {m_{na}} + \frac{3}{2}} \right)}} \hfill \nonumber \\
   &\times F\left( {2m + 2,m - {m_{na}} + \frac{3}{2};m + {m_{na}} + \frac{3}{2};{\Upsilon _s}} \right),
\end{align}
\end{small}where ${\Upsilon _s} = \frac{{s - 2\sqrt {\left( {\kappa  + 1} \right){m_{na}}} }}{{s + 2\sqrt {\left( {\kappa  + 1} \right){m_{na}}} }}$.
When ${{P_s} \to \infty }$ in the time domain, $s$ tends to infinity in the complex frequency domain simultaneously.
Therefore, $\frac{{s - 2\sqrt {\left( {\kappa  + 1} \right){m_{na}}} }}{{s + 2\sqrt {\left( {\kappa  + 1} \right){m_{na}}} }}$ can be approximated to 1 and $s + 2\sqrt {\left( {\kappa  + 1} \right){m_{na}}}$ can be approximated as $s$.
Letting $\Delta \left( m \right)$ as
\begin{small}
\begin{align}\label{detam}
  \Delta \left( m \right) &= \frac{{{\kappa ^m}{{\left[ {\left( {\kappa  + 1} \right){m_{na}}} \right]}^{\frac{1}{2}\left( {1 + m + {m_{na}}} \right)}}}}{{{{\left( {m!} \right)}^2}\Gamma \left( {m + {m_{na}} + \frac{3}{2}} \right)}} \hfill \nonumber \\
   &\times {\left( {4\sqrt {\left( {\kappa  + 1} \right){m_{na}}} } \right)^{m - {m_{na}} + 1}}\Gamma \left( {2m + 2} \right) \hfill \nonumber \\
   &\times F\left( {2m + 2,m - {m_{na}} + \frac{3}{2};m + {m_{na}} + \frac{3}{2};1} \right),
\end{align}
\end{small}
the Laplace transform of ${X_n} = \left| {h_n^1h_n^2} \right|$ can be eventually derived as
\begin{align}\label{OP for Dn with pSIC step3}
\mathcal{L}\left[ {{f_{{X_n}}}\left( y \right)} \right]\left( s \right) = \sum\limits_{m = 0}^\infty  {\frac{{4\sqrt \pi  \Gamma \left( {2{m_{na}}} \right)\Delta \left( m \right)}}{{{e^\kappa }\Gamma \left( {{m_{na}}} \right)}}} \frac{1}{{{s^{2m + 2}}}}.
\end{align}

As a further advance, based on \eqref{OP for Dn with pSIC step3} and then using the convolution theorem, the Laplace transform of the PDF of ${{{\hat \zeta }_h}}$ can be given by
\begin{align}\label{OP for Dn with pSIC step4}
\mathcal{L}\left[ {{f_{{{\hat \zeta }_h}}}\left( y \right)} \right]\left( s \right) = {\left[ {\sum\limits_{m = 0}^\infty  {\frac{{4\sqrt \pi  \Gamma \left( {2{m_{na}}} \right)\Delta \left( m \right)}}{{{e^\kappa }\Gamma \left( {{m_{na}}} \right)}}} \frac{1}{{{s^{2m + 2}}}}} \right]^N}.
\end{align}

The first term of the series is taken i.e.,$m=0$, and then appling inverse Laplace transform to \eqref{OP for Dn with pSIC step4}, the PDF of ${{{\hat \zeta }_h}}$ at high SNR region can be derived as
\begin{align}\label{OP for Dn with pSIC step5}
{f_{{{\hat \zeta }_h}}}\left( y \right) = {\left[ {\frac{{4\sqrt \pi  \Gamma \left( {2{m_{na}}} \right)\Delta \left( 0 \right)}}{{{e^\kappa }\Gamma \left( {{m_{na}}} \right)}}} \right]^N}\frac{{{y^{2N - 1}}}}{{\left( {2N - 1} \right)!}}.
\end{align}

Similarly, the derivation of ${{{\hat \zeta }_g}}$ follows the same process as above.
Referring to the process of \eqref{OP for Dn with ipSIC step3} and using the transformation of the Gaussian Chebyshev quadrature formula\cite[Eq. (8.8.12)]{Introduction_to_Numerical_Analysis}, the CDF expression for $\hat Z = {{\hat \zeta }_h}{{\hat \zeta }_g}$ at high SNR region can be obtained as
\begin{small}
\begin{align}\label{OP for Dn with pSIC step7}
  {F_{\hat Z}}\left( z \right) &= \frac{{\pi {\mu _b}{z^{2M}}}}{{2K\left( {2M} \right)!\left( {2N - 1} \right)!}}{\left[ {\frac{{4\sqrt \pi  \Gamma \left( {2{m_{na}}} \right)\Delta \left( 0 \right)}}{{{e^\kappa }\Gamma \left( {{m_{na}}} \right)}}} \right]^{M + N}} \hfill \nonumber \\
   &\times \sum\limits_{k = 1}^K {{{\left[ {\frac{{\left( {{x_k}{\text{ + }}1} \right){\mu _b}}}{2}} \right]}^{2N - 2M - 1}}\sqrt {1 - {x_k}^2} },
\end{align}
\end{small}where $K$ is the parameter of the Gauss-Chebyshev quadrature formula for precision.
$\mu _b$ is the upper limit of integral before using the Gauss-Chebyshev quadrature formula for the approximation and ${\mu _b} \to \infty $. ${x_k} = \cos \left( {\frac{{2k - 1}}{{2K}}\pi } \right)$.

Finally, by submitting \eqref{OP for Dn with pSIC step7} into \eqref{OP for Dn with pSIC step1}, \eqref{asymptotic outage probability of Dn with pSIC} can be obtained.
The proof is completed.
\section*{Appendix~C: Proof of Theorem \ref{theorem:4}} \label{Appendix:C}
\renewcommand{\theequation}{C.\arabic{equation}}
\setcounter{equation}{0}
The proof starts by substituting \eqref{SINRN} into \eqref{ergodic rate} and we can obtain the expression of ergodic data rate for $D_n$ with ipSIC as
\begin{small}
\begin{align}\label{ER for Dn with ipSIC step1}
R_{Dn,ipSIC}^{ergodic} = \mathbb{E}\left[ {{{\log }_2}\left( {1 + \frac{{\beta {P_s^{a}}{{\left| {{\mathbf{h}}_2^T{{\mathbf{\Phi }}_p}{{\mathbf{h}}_1}} \right|}^2}{{\left| {{\mathbf{g}}_n^T{{\mathbf{\Phi }}_a}{{\mathbf{g}}_1}} \right|}^2}{a_n}}}{{\varpi {P_s^{a}}\sigma _{ri}^2 + \beta \sigma _a^2{{\left\| {{\mathbf{g}}_n^T{{\mathbf{\Phi }}_a}} \right\|}^2} + {\sigma ^2}}}} \right)} \right].
\end{align}
\end{small}

By referring to the process of proving the outage probability of $D_n$ with ipSIC in \textbf{Appendix A} and with the use of \eqref{OP for Dn with ipSIC}, we can obtain the CDF of $\gamma _{{D_n}}$ as
\begin{align}\label{ER for Dn with ipSIC step2}
  &{F_{{\gamma _{{D_n}}}}}\left( z \right) = \sum\limits_{i = 1}^I {\sum\limits_{u = 1}^U {\frac{{{H_i}{H_u}x_u^{{a_X}}}}{{\Gamma ({a_X} + 1)\Gamma ({a_Y} + 1)}}} }  \hfill \nonumber \\
   &\times \gamma \left( {{a_Y} + 1,\frac{{\sqrt {z{{\bar \lambda }_n}\left( {{P_s^{a}}{\Omega _{RI}}{x_i} + \beta \sigma _a^2M{\Omega _{na}} + {\sigma ^2}} \right)} }}{{\sqrt {\beta {P_s^{a}}{a_n}} {b_X}{b_Y}{x_u}}}} \right).
\end{align}

With the help of the definition expression for the lower incomplete Gamma function, we can obtain the PDF of $\gamma _{{D_n}}$ on the basis of \eqref{ER for Dn with ipSIC step2} as
\begin{align}\label{ER for Dn with ipSIC step3}
{f_{{\gamma _{{D_n}}}}}\left( z \right) = \sum\limits_{i = 1}^I {\sum\limits_{u = 1}^U {\frac{{{a^{{a_Y} + 1}}{H_i}{H_u}x_u^{{a_X}}{z^{\frac{{{a_Y}}}{2}}}{e^{ - \chi \sqrt z }}}}{{2\sqrt z \Gamma ({a_X} + 1)\Gamma ({a_Y} + 1)}}} },
\end{align}
where $\chi  = \frac{{\sqrt {z{{\bar \lambda }_n}\left( {{P_s}{\Omega _{RI}}{x_i} + \beta \sigma _a^2M{\Omega _{na}} + {\sigma ^2}} \right)} }}{{\sqrt {\beta {P_s}{a_n}} {b_X}{b_Y}{x_u}}}$.

Combining \eqref{ER for Dn with ipSIC step1} and \eqref{ER for Dn with ipSIC step3}, the ergodic data rate of $D_n$ with ipSIC for PRIS-ARIS-NOMA can be given by
\begin{small}
\begin{align}\label{ER for Dn with ipSIC step4}
  &R_n^{ergodic} = \frac{1}{{\ln 2}}\int_0^{ + \infty } {\ln \left( {1 + y} \right){f_{{\gamma _{{D_n}}}}}\left( y \right)} dy \hfill \nonumber \\
   &= \frac{1}{{\ln 2}}\int_0^{ + \infty } {\ln \left( {1 + y} \right)} \sum\limits_{i = 1}^I {\sum\limits_{u = 1}^U {\frac{{{a^{{a_Y} + 1}}{H_i}{H_u}x_u^{{a_X}}{z^{\frac{{{a_Y}}}{2}}}{e^{ - \chi \sqrt z }}}}{{2\sqrt z \Gamma ({a_X} + 1)\Gamma ({a_Y} + 1)}}} } dy.
\end{align}
\end{small}

Finally, by bringing into Gauss-Laguerre quadrature, \eqref{ergodic for Dn with ipSIC} can be obtained.
The proof is completed.
\section*{Appendix~D: Proof of Theorem \ref{theorem:5}} \label{Appendix:D}
\renewcommand{\theequation}{D.\arabic{equation}}
\setcounter{equation}{0}
The proof starts similarly to the \eqref{ER for Dn with ipSIC step1} by substituting \eqref{SINRD} into \eqref{ergodic rate} and we can obtain the expression of ergodic data rate for $D_m$ as
\begin{small}
\begin{align}\label{ER for Dm step1}
  R_{{D_m}}^{ergodic} &= \mathbb{E}\left[ {{{\log }_2}\left( {1 + } \right.} \right. \hfill \nonumber \\
  &\left. {\left. {\frac{{\beta {P_s^{a}}{{\left| {{\mathbf{h}}_2^T{{\mathbf{\Phi }}_p}{{\mathbf{h}}_1}} \right|}^2}{{\left| {{\mathbf{g}}_m^T{{\mathbf{\Phi }}_a}{{\mathbf{g}}_1}} \right|}^2}{a_m}}}{{\beta {P_s^{a}}{{\left| {{\mathbf{h}}_2^T{{\mathbf{\Phi }}_p}{{\mathbf{h}}_1}} \right|}^2}{{\left| {{\mathbf{g}}_m^T{{\mathbf{\Phi }}_a}{{\mathbf{g}}_1}} \right|}^2}{a_n} + \beta \sigma _a^2{{\left\| {{\mathbf{g}}_n^T{{\mathbf{\Phi }}_a}} \right\|}^2} + {\sigma ^2}}}} \right)} \right].
\end{align}
\end{small}

By referring to the process of proving the outage probability of $D_m$ and with the use of \eqref{OP for Dm}, we can obtain the CDF of $\gamma _{{D_m}}$ as
\begin{small}
\begin{align}\label{ER for Dm step2}
  {F_{{{\gamma _{{D_m}}}}}}\left( x \right) =& \sum\limits_{u = 1}^U {\frac{{x_u^{{a_X}}{H_u}}}{{\Gamma ({a_X} + 1)\Gamma ({a_Y} + 1)}}}  \hfill \nonumber \\
   &\times \gamma \left( {{a_Y} + 1,\frac{{\sqrt {\overline {{\lambda _m}} x\left( {\beta \sigma _a^2M{\Omega _{na}} + {\sigma ^2}} \right)} }}{{\sqrt {\beta {P_s^{a}}\left( {{a_m} - x{a_n}} \right)} {b_X}{b_Y}{x_u}}}} \right),
\end{align}
\end{small}
it should be noted that in order to ensure the correctness of the calculation, it is necessary here to satisfy $x < \frac{{{a_m}}}{{{a_n}}}$, which is different from the process of solving for \eqref{ER for Dn with ipSIC step2}.
Combining \eqref{ER for Dm step1} and \eqref{ER for Dm step2}, the ergodic data rate of $D_m$ for PRIS/ARIS-NOMA can be given by
\begin{align}\label{ER for Dm step3}
  R_m^{ergodic} =& \frac{1}{{\ln 2}}\int_0^{\frac{{{a_m}}}{{{a_n}}}} {\frac{{1 - {F_{{\gamma _{{D_m}}}}}\left( x \right)}}{{1 + x}}dx},
\end{align}

Finally, by bringing into the transformation of the Gaussian Chebyshev quadrature formula, \eqref{ergodic for Dm} can be obtained.
The proof is completed.
\bibliographystyle{IEEEtran}
\bibliography{mybib}

\end{document}